\documentclass[amscd,12pt,a4paper]{amsart}

\usepackage{amscd,amsmath,amssymb,amsfonts}

\numberwithin{equation}{section}
\theoremstyle{plain}
   \newtheorem{thm}{Theorem}[section]
   \newtheorem{lem}[thm]{Lemma}
   
   \newtheorem{prop}[thm]{Proposition}
   \newtheorem{cor}[thm]{Corollary}
   \newtheorem{conj}[thm]{Conjecture}
   \newtheorem{claim}[thm]{Claim}
\theoremstyle{definition}
   \newtheorem{defn}[thm]{Definition}
   
   \newtheorem{ex}[thm]{Example}
\theoremstyle{remark}
   \newtheorem{rem}[thm]{Remark}
\def\cd{{\mathrm{cd}}}      
\def\CH{{\mathrm{CH}}}      
\def\Coker{{\mathrm{Coker}}}
\def\Ker{{\mathrm{Ker}}}    
\def\Br{{\mathrm{Br}}}    
\def\et{{\text{\'{e}t}}}    
\def\Gal{{\mathrm{Gal}}}    
\def\H{{\mathrm{H}}}        
\def\Hom{{\mathrm{Hom}}}    
\def\Image{{\mathrm{Image}}}   

\def\Spec{{\mathrm{Spec}}}
\def\Hot{{\mathrm{Hot}}}
\def\Ind{{\mathrm{Ind}}}
\def\ch{{\mathrm{ch}}}    

\def\cF{{\mathcal F}}
\def\cX{{\mathcal X}}
\def\cO{{\mathcal O}}
\def\cA{{\mathcal A}}
\def\cC{{\mathcal C}}
\def\cZ{{\mathcal Z}}
\def\cS{{\mathcal S}}

\def\cK{{\mathcal K}}

\def\ZS{\bZ\cS}
\def\CorS{Cor\cS}

\def\Lam{\Lambda}
\def\lam{\lambda}

\def\k{\kappa}

\def\fm{{\frak{m}}}          
\def\fp{{\frak{p}}}          

\def\wt#1{\widetilde{#1}}

\def\ol#1{\overline{#1}}

\def\us#1#2{\underset{#1}{#2}}

\def\lim#1{\us{#1}{\varinjlim}}
\def\sumd#1#2{\underset{x\in {#1}_{#2}}\bigoplus}

\def\bR{{\mathbb R}}
\def\bZ{{\mathbb Z}}
\def\bQ{{\mathbb Q}}

\def\bL{{\mathbb L}}

\def\qz{{\bQ}/{\bZ}}
\def\qzl{{\bQ_\ell}/{\bZ_\ell}}
\def\qzp{{\bQ_p}/{\bZ_p}}
\def\zl{{\bZ}_\ell}

\def\qlz{{\bQ}_\ell/{\bZ}_\ell}
\def\zp{{\bZ}_p}

\def\pnz{\bZ/p^n\bZ}

\def\nz{\bZ/n\bZ}
\def\mz{\bZ/m\bZ}
\def\lz{\bZ/\ell\bZ}
\def\lnz{\bZ/\ell^n\bZ}

\def\indlim#1{\underset{{\underset{#1}{\longrightarrow}}}{\mathrm{lim}}\; }
\def\projlim#1{\underset{{\underset{#1}{\longleftarrow}}}{\mathrm{lim}}\; }

\def\rmapo#1{\overset{#1}{\longrightarrow}}

\def\rmapou#1#2{\overset{#1}{\underset{#2}{\longrightarrow}}}
\def\lmapo#1{\overset{#1}{\longleftarrow}}
\def\isom{\overset{\cong}{\longrightarrow}}

\def\tX{\tilde{X}}
\def\tY{\tilde{Y}}

\def\qaq{\quad\text{and}\quad}
\def\qfor{\quad\text{for }}
\def\qforall{\quad\text{for all }}
\def\qwith{\quad\text{with }}

\def\Kb{\ol K}

\def\otkb#1{{#1}_{\Kb}}

\def\Ln{\Lambda_n}

\def\Linfty{\Lambda_\infty}
\def\Mod{{Mod}}

\def\tX{\widetilde{X}}
\def\tY{\widetilde{Y}}
\def\tZ{\widetilde{Z}}
\def\tU{\widetilde{U}}

\def\EX#1#2#3{E^{#3}_{#1,#2}(X)}
\def\ZX#1#2#3{Z^{#3}_{#1,#2}(X)}
\def\BX#1#2#3{B^{#3}_{#1,#2}(X)}
\def\dX#1#2#3{d^{#3}_{#1,#2}(X)}
\def\ZBX#1#2#3#4{(Z^{#3}/B^{#4})_{#1,#2}(X)}
\def\KX#1#2#3{K^{#3}_{#1,#2}(X)}
\def\HXZ#1{H_{p+q}(\cZ_{#1}(X))}
\def\HXZa#1{H_{p+q-1}(\cZ_{#1}(X))}

\def\piX#1#2{\pi_{#1,#2}}

\def\delX#1#2{\delta_{#1,#2}}

\def\EU#1#2#3{E^{#3}_{#1,#2}(U)}
\def\ZU#1#2#3{Z^{#3}_{#1,#2}(U)}
\def\BU#1#2#3{B^{#3}_{#1,#2}(U)}

\def\ZBU#1#2#3#4{(Z^{#3}/B^{#4})_{#1,#2}(U)}

\def\ZBUn#1#2#3#4{(Z^{#3}/B^{#4})_{#1,#2}(U,\Ln)}

\def\EUnz#1#2#3{E^{#3}_{#1,#2}(U,\nz)}

\def\ZBUZ#1#2#3#4{(Z^{#3}/B^{#4})_{#1,#2}(U\backslash Z)}
\def\ZBZcapU#1#2#3#4{(Z^{#3}/B^{#4})_{#1,#2}(Z\cap U)}

\def\EY#1#2#3{E^{#3}_{#1,#2}(Y)}
\def\ZY#1#2#3{Z^{#3}_{#1,#2}(Y)}
\def\BY#1#2#3{B^{#3}_{#1,#2}(Y)}
\def\dY#1#2#3{d^{#3}_{#1,#2}(Y)}
\def\ZBY#1#2#3#4{(Z^{#3}/B^{#4})_{#1,#2}(Y)}
\def\KY#1#2#3{K^{#3}_{#1,#2}(Y)}
\def\HYZ#1{H_{p+q}(\cZ_{#1}(Y))}
\def\HYZa#1{H_{p+q-1}(\cZ_{#1}(Y))}
\def\HYZZ#1{H_{p+q}(\cZ_{#1}/\cZ_{#1-1}(Y))}
\def\piY#1#2{\pi_{#1,#2}}

\def\delY#1#2{\delta_{#1,#2}}

\def\EV#1#2#3{E^{#3}_{#1,#2}(V)}
\def\ZV#1#2#3{Z^{#3}_{#1,#2}(V)}
\def\BV#1#2#3{B^{#3}_{#1,#2}(V)}
\def\dV#1#2#3{d^{#3}_{#1,#2}(V)}
\def\ZBV#1#2#3#4{(Z^{#3}/B^{#4})_{#1,#2}(V)}
\def\BBV#1#2#3#4{(B^{#3}/B^{#4})_{#1,#2}(V)}
\def\KV#1#2#3{K^{#3}_{#1,#2}(V)}
\def\HVZ#1{H_{p+q}(\cZ_{#1}(V))}

\def\piV#1#2{\pi_{#1,#2}}

\def\delV#1#2{\delta_{#1,#2}}

\def\HstXZZ#1{H_{s+t}(\cZ_{#1}/\cZ_{#1-1}(X))}
\def\HstYZZ#1{H_{s+t}(\cZ_{#1}/\cZ_{#1-1}(Y))}
\def\HstVZ#1{H_{s+t}(\cZ_{#1}(V))}
\def\HstXZ#1{H_{s+t}(\cZ_{#1}(X))}
\def\HstoYZ#1{H_{s+t-1}(\cZ_{#1}(Y))}

\def\FF#1#2#3{\Phi_{#3}^{#1,#2}}
\def\tFF#1#2#3{\widetilde{\Phi}_{#3}^{#1,#2}}
\def\gF#1#2#3{g^{#1,#2}_{#3}}
\def\fF#1#2#3{f^{#1,#2}_{#3}}

\def\psiF#1#2#3{\psi^{#1,#2}_{#3}}
\def\tpsiF#1#2#3{\widetilde{\psi}^{#1,#2}_{#3}}

\def\Hempty#1#2{H_{#1}(#2)}

\def\HLLinfty{H(-,\Linfty)}
\def\HDL#1#2{H^D_{#1}(#2,\Lambda)}
\def\HDK#1#2{H^D_{#1}(#2,\cK)}

\def\HLn#1#2{H_{#1}(#2,\Ln)}

\def\HLLn{H(-,\Ln)}
\def\HDLn#1#2{H^D_{#1}(#2,\Ln)}
\def\HDLinfty#1#2{H^D_{#1}(#2,\Linfty)}

\def\Hn#1#2{H_{#1}(#2,\Ln)}

\def\Hinfty#1#2{H_{#1}(#2,\Linfty)}
\def\Hetinfty#1#2{H^{\et}_{#1}(#2,\Linfty)}

\def\KC#1{KC_H(#1)}
\def\KCLH#1{KC_H(#1,\Lambda_H)}

\def\KH#1#2{KH_{#1}(#2)}

\def\KHn#1#2{KH_{#1}(#2,\Lambda_n)}

\def\KHinfty#1#2{KH_{#1}(#2,\Lambda_\infty)}

\def\KHetnz#1#2{KH^{\et}_{#1}(#2,\nz)}

\def\KHetpnz#1#2{KH^{\et}_{#1}(#2,\pnz)}
\def\KHetlnz#1#2{KH^{\et}_{#1}(#2,\lnz)}

\def\KHetqlz#1#2{KH^{\et}_{#1}(#2,\qlz)}
\def\KHetqz#1#2{KH^{\et}_{#1}(#2,\qz)}

\def\graph#1{Graph(#1,\Lambda_H)}
\def\graphL#1{Graph(#1,\Lambda)}

\def\graphH#1#2{Graph_{#1}(#2,\Lambda_H)}
\def\graphHempty#1#2{Graph_{#1}(#2,\Lambda_H)}
\def\graphHempty#1#2{Graph_{#1}(#2)}
\def\graphHn#1#2{Graph_{#1}(#2,\Lambda_n)}
\def\graphHinfty#1#2{Graph_{#1}(#2,\Lambda_\infty)}
\def\graphHnz#1#2{Graph_{#1}(#2,\nz)}

\def\graphhom#1#2{\gamma^{#1}_{#2}}

\def\graphedge#1#2{\gamma\epsilon^{#1}_{#2}}
\def\graphedgen#1#2{\gamma\epsilon^{#1}_{#2,\Ln}}

\def\edgehom#1{\epsilon_{#1}}

\def\DWr#1#2{W_r\Omega^{#1}_{#2}}
\def\DWrlog#1#2{W_r\Omega^{#1}_{#2,log}}
\def\DWt#1#2{W_t\Omega^{#1}_{#2}}
\def\DWtlog#1#2{W_t\Omega^{#1}_{#2,log}}

\def\DWnulog#1#2{W_\nu\Omega^{#1}_{#2,log}}

\def\Th#1#2{\Theta_{#1}(#2)}

\def\WC#1{W_\bullet(#1)}
\def\tWC#1{\tilde{W}_\bullet(#1)}
\def\wtWC#1{\widetilde{W}_\bullet(#1)}

\def\ol#1{\overline{#1}}

\def\PK{P_K}
\def\Pm{P_\fm}

\def\BKlz#1{{\text{\bf (BK)}_{#1,\ell}}}

\def\Dql{{\text{\bf (D)}_{q,\ell}}}

\def\RS#1{{\text{\bf (RS)}_{#1}}}
\def\RES#1{{\text{\bf (RES)}_{#1}}}

\def\HMXr#1{H^{#1}_M(X,\bZ(r))}

\def\HMXd#1{H^{#1}_M(X,\bZ(d))}

\def\nt{\ell^n}
\def\mt{\ell^m}

\begin{document}

\author{Uwe Jannsen and Shuji Saito}
\title[Kato conjecture and motivic cohomology over finite fields]
{Kato conjecture and motivic cohomology\\
 over finite fields}
\maketitle

\pagenumbering{roman}

\setcounter{page}{1}
\tableofcontents
\begin{abstract}
For an arithmetical scheme $X$, K. Kato introduced a certain complex of
Gersten-Bloch-Ogus type whose component in degree $a$ involves Galois
cohomology groups of the residue fields of all the points of $X$ of dimension
$a$. He stated a conjecture on its homology generalizing the fundamental
exact sequences for Brauer groups of global fields. We prove the conjecture
over a finite field assuming resolution of singularities.
Thanks to a recently established result on resolution of singularities
for embedded surfaces, it implies the unconditional vanishing of
the homology up to degree $4$ for $X$ projective smooth over a finite field.
We give an application to finiteness questions
for some motivic cohomology groups over finite fields.
\end{abstract}

\smallskip
\noindent\small{{2010 Mathematics Subject Classification$:$ 14F20, 14F42, 14G15}\\
{Keywords$:$ Kato conjecture, varieties over finite fields, motivic cohomology}
}

\tableofcontents
\pagenumbering{arabic}

\section*{Introduction}
\bigskip

Let $\cC$ be a category of schemes of finite type over
a fixed base scheme $B$.
Let $\cC_\ast$ be the category with the same objects as $\cC$, but where
morphisms are just the proper morphisms in $\cC$.
Let $\Mod$ is the category of modules.
A \it homology theory $H=\{H_a\}_{a\in \bZ}$ on $\cC$ \rm is
a sequence of covariant functors:
$$
H_a(-):~\cC_*\rightarrow \Mod
$$
satisfying the following conditions:
If $i:Y\hookrightarrow X$ is a closed immersion in $\cC$,
with open complement $j:V\hookrightarrow X$, there is a long exact sequence
(called localization sequence)
$$
\cdots\rmapo{\partial} H_a(Y) \rmapo{i_*} H_a(X) \rmapo{j^*} H_a(V) \rmapo{\partial} H_{a-1}(Y)
\longrightarrow \cdots.
$$
which is functorial with respect to proper morphisms and open immersions in an obvious sense.
\medbreak

Given such $H$, Bloch and Ogus \cite{BO} constructed a spectral sequence of
homological type for every $X \in Ob(\cC)$, called the niveau spectral sequence
\begin{equation}\label{spectralsequence1}
E^1_{a,b}(X)=\bigoplus_{x\in X_a}H_{a+b}(x)~~\Rightarrow ~~H_{a+b}(X) \qwith
H_d(x)=\lim{V\subseteq \overline{\{x\}}} H_d(V).
\end{equation}
where $X_a$ denotes the set of the points $x\in X$ of dimension $a$
(see \S1 for our definition of dimension) and the limit is over all open
non-empty subschemes $V\subseteq \overline{\{x\}}$.
\medbreak

In this paper we are interested in such $H$ that satisfy the condition:
\begin{equation}\label{conditionKatohomology}
E^1_{a,b}(X)=0 \qforall b<0\text{ and all  } X\in Ob(\cC).
\end{equation}
We then write $\KC X$ for the complex:
$$
E^1_{\bullet,0}(X)\;:\; E^1_{0,0}(X) @<{d^1}<< E^1_{1,0}(X) @<{d^1}<<
E^1_{2,0}(X) @<{d^1}<< \cdots
$$
and $\KH a X$ for its homology group in degree $a$,
called the Kato homology of $X$ for the given homology theory $H$.
\medbreak

The most typical example is the homology theory $H = H^{\et}(-,\nz)[-1]$ on
the category $\cC$ of separated schemes over a finite field $F$, defined by
\begin{equation}\label{etalehomology}
H^\et_a(X,\nz)[-1]:=H^{1-a}(X_{\acute{e}t}, R\,f^{!}\nz),
\qfor f:X\rightarrow \Spec(F) \text{ in }\cC
\end{equation}
(where $Rf^!$ is the right adjoint of $Rf_!$ defined in [SGA 4], XVIII, 3.1.4.).
In this case $\KC X$ is the following complex introduced by Bloch-Ogus
\cite{BO} and Kato \cite{K}:
\begin{multline}\label{KCfinitefieldintro}
\cdots \sumd X a H^{a+1}(\k(x),\nz(a))\to
\sumd X {a-1} H^{a}(\k(x),\nz(a-1))\to \cdots \\
\cdots \to\sumd X 1 H^{2}(\k(x),\nz(1))\to
\sumd X 0 H^{1}(\k(x),\nz).
\end{multline}
Here we use the following notations.
For a field $L$ and an integer $n>0$, define the following
Galois cohomology groups: If $n$ is invertible in $L$, let
$H^i(L,\nz(j))=H^i(L,\mu_n^{\otimes j})$ where $\mu_n$ is the Galois module of
$n$-th roots of unity.
If $n=mp^r$ and $(p,m)=1$ with $p=\ch(L)>0$, let
\begin{equation}\label{def.nz}
H^i(L,\nz(j))=H^i(L,\mz(j))\oplus H^{i-j}(L,\DWrlog j L),
\end{equation}
where $\DWrlog j L$ is the logarithmic part of the de Rham-Witt sheaf
$\DWr j L$ [Il, I 5.7].
In the complex \eqref{KCfinitefieldintro} the term in degree $a$ is the direct
sum of the Galois cohomology of the residue fields $\k(x)$ of $x\in X_a$.
\medbreak

In \cite{K} a complex of the same shape is defined for any scheme $X$
of finite type over $\Spec(\bZ)$ and it is shown in \cite{JSS}
that this complex also arises from a certain homology theory
(on the category of schemes of finite type over $\Spec(\bZ)$)
via the associated spectral sequence \eqref{spectralsequence1}.
\medbreak

The Kato homology associated to $H = H^{\et}(-,\nz)[-1]$
is denoted by
$\KHetnz a X$, which is by definition the homology group in degree $a$ of
the complex \eqref{KCfinitefieldintro}.
A remarkable conjecture proposed by Kato is the following:

\begin{conj}\label{conjKato}
Let $X$ be either proper smooth over $B=\Spec(F)$ where $F$ is
a finite field (geometric case), or regular proper flat over $B=\Spec(\cO_k)$
where $k$ is a number field (arithmetic case).
Assume either $n$ is odd or $k$ is totally imaginary. Then
$$
\KHetnz a X \simeq
\left.\left\{\gathered
 \nz \\
 0 \\
\endgathered\right.\quad
\aligned
&\text{$a=0$}\\
&\text{$a\not=0$}
\endaligned\right.
$$
\end{conj}
\medbreak

In case $\dim(X)=1$, i.e., if $X$ is a proper smooth curve over a finite field
with function field $k$ or $X=\Spec(\cO_k)$ as above,
the conjecture \ref{conjKato} rephrases the classical fundamental fact
in number theory that there is an exact sequence :
$$
0\to \Br(k)[n] \to \underset{x\in X_0}{\bigoplus}
 H^1(\k(x),\nz) \to \nz \to 0.
$$
Here $\Br(k)[n]$ is the $n$-torsion subgroup of the Brauer group of $k$
and $X_0$ is the set of the closed points of $X$ or the finite places of $k$.
Kato proved the conjecture in case $\dim(X)=2$.
The following result has been shown by
Colliot-Th\'el\`ene \cite{CT} and Suwa \cite{Sw} (geometric case) and
Jannsen-Saito \cite{JS1} (arithmetic case):

\begin{thm}\label{thm.CTSJS}
Let $\ell$ be a rational prime.
Let the assumption be as in \ref{conjKato} and assume $X$ is projective
over $B$. In the arithmetic case we further assume $X$ has good or semistable
reduction at each prime of $\cO_k$ and that $\ell$ is odd or $k$ is
totally imaginary. Then
$$
\KHetqlz a X  \simeq \left.\left\{\gathered
 \qzl \\
 0 \\
\endgathered\right.\quad
\aligned
&\text{$a=0$}\\
&\text{$0< a \leq 3$}
\endaligned\right.
$$
where
$$
\KHetqlz a X =\indlim n \KHetlnz a X.
$$
\end{thm}

In this paper we propose a method to approach the geometric case of
conjecture \ref{conjKato} in general.
The main result of this paper is the following:

\begin{thm}\label{mainthmintro1}
Let $X$ be projective smooth of dimension $d$ over a finite field $F$.
Let $t\geq 1$ be an integer.
Then we have
$$
\KHetqz a X  \simeq \left.\left\{\gathered
 \qz \\
 0 \\
\endgathered\right.\quad
\aligned
&\text{$a=0$}\\
&\text{$0< a \leq t$}
\endaligned\right.
$$
if either $t\leq 4$ or condition {$\RS {d}$}, or {$\RES {t-2}$} (see below) holds.
Moreover the same conclusion holds if
$\qz$ is replaced by $\lnz$ for a prime $\ell$, provided $\BKlz t$ holds (see below).
\end{thm}
\medbreak

Now we explain the conditions used in the above theorem. The first two concern resolution of
singularities:
\begin{enumerate}
\item[{$\RS d$}] :
For any $X$ integral and proper of dimension $\leq d$ over $F$,
there exists a proper birational morphism
$\pi: X'\to X$ such that $X'$ is smooth over $F$.
For any $U$ smooth of dimension $\leq d$ over $F$, there is an open immersion
$U\hookrightarrow X$ such that $X$ is projective smooth over $F$ with $X-U$,
a simple normal crossing divisor on $X$.
\end{enumerate}

\begin{enumerate}
\item[{$\RES t$}] :
For any smooth projective variety $X$ over $F$,
any simple normal crossing divisor $Y$ on $X$ with $U=X-Y$,
and any integral closed subscheme $W\subset X$ of dimension$\leq t$
such that $W\cap U$ is regular,
there exists a smooth projective variety $X'$ over $F$ and a birational proper map
$\pi\,:\, X'\to X$ such that
$\pi\,:\,\pi^{-1}(U)\simeq U$, and $Y'=X'-\pi^{-1}(U)$ is
a simple normal crossing divisor on $X'$,
and the proper transform of $W$ in $X'$ is regular and intersects transversally
with $Y'$.
\end{enumerate}
\medbreak

We note that a proof of {$\bf (RES)_2$} is given in \cite{CJS} based on
ideas of Hironaka. This enables us to obtain the unconditional vanishing
of the Kato homology with $\qz$-coefficient in degrees $a\leq 4$ in Theorem \ref{mainthmintro1}.
\bigskip

Let $\ell$ be a prime and $L$ be a field.
Recall that there is a symbol map (\cite{Mi} and \cite{BK}, \S2):
$$h^t_{L,\ell} : K^M_t(L) \to H^t(L,\lz(t))$$
where $K^M_t(L)$ denotes the Milnor $K$-group of $L$.
It is conjectured that $h^t_{L,\ell}$ is surjective.
The conjecture is called the Bloch-Kato conjecture in case $l\not=\ch(L)$.
We introduce the following condition:
\medbreak

\begin{enumerate}
\item[{$\BKlz {t}$}] :
For any finitely generated field $L$ over $F$,
$h^t_{L,\ell}$ is surjective.
\end{enumerate}
\medbreak

The surjectivity of $h^t_{L,\ell}$ is known if
$t=1$ (the Kummer theory) or
$t=2$ (Merkurjev-Suslin \cite{MS}) or
$\ell =\ch(L)$ (Bloch-Gabber-Kato \cite{BK}) or
$\ell =2$ (Voevodsky \cite{V1}).
It has been announced by Rost and Voevodsky (\cite{SJ} and \cite{V2})
that it holds in general.
\bigskip

In fact Theorem \ref{mainthmintro1} will be deduced from the following
more general result:

\begin{thm}\label{mainthmintro2}
Let $H$ be a homology theory on the category $\cC$ of separated schemes over
$B=\Spec(k)$ for a field $k$, which satisfies \eqref{conditionKatohomology}.
Assume the following conditions:
\begin{itemize}
\item[$(H1)$]
If $f:X\to B=\Spec(k)$ is smooth projective of dimension $\leq 1$
with $X$ connected (but not necessarily geometrically irreducible over $k$), then
$$
f_*\;:\; \Hempty {0} X \to \Hempty {0} {B}
$$
is an isomorphism if $\dim(X)=0$ and injective if $\dim(X)=1$.
\item[$(H2)$]
If $X$ is smooth projective of dimension $>1$ over $B$,
$Y\subset X$ is an irreducible smooth ample divisor, and $U=X-Y$, then
$$ \Hempty {a} {U} =0\qfor a\leq d=\dim(X).$$
\item[$(H3)$]
If $X$ is a smooth projective curve over $B$ and $U\subset X$ is a dense affine open
subset, then
$$ \Hempty {a} {U} =0\qfor a\leq 0\,,$$
and the boundary map $\Hempty {1} U @>{\partial}>> \Hempty {0} Y$ is
injective, where
$Y=X-U$ with reduced subscheme structure.
\end{itemize}
Let $X$ be projective smooth of dimension $d$ over $B$.
Let $t\geq 1$ be an integer with $t\leq d$.
Assume either $t\leq 4$ or $\bf (RS)_{d}$ or $\bf (RES)_{t-2}$.
Then we have
$$ \KH a X = 0 \qforall 0< a\leq t.$$
\end{thm}
\medbreak

Theorem \ref{mainthmintro1} follows from \ref{mainthmintro2} by verifying
that the homology theory $H = H^{\et}(-,\qz)[-1]$ (see \eqref{etalehomology})
satisfies the conditions of \ref{mainthmintro2}. This is done by using
the affine Lefschetz theorem and the Weil conjecture proved by Deligne \cite{D}.
We will also give an example of a homology theory other than
$H^{\et}(-,\qz)[-1]$, which satisfies the conditions of \ref{mainthmintro2}
(see Lemma \ref{lemL2}).
\bigskip

In what follows we explain an application of Theorem \ref{mainthmintro1}
to finiteness result for motivic cohomology of smooth schemes
over a finite field.

Let $X$ be a connected smooth scheme over a finite field $F$ and let
$$\HMXr q=  CH^r(X,2r-q) = H_{2r-q}(z^r(X,\bullet))$$
be the motivic cohomology of $X$ defined as Bloch's higher Chow group,
where $z^r(X,\bullet)$ is Bloch's cycle complex \cite{B1}.
We will review the definition in \S6.
A `folklore conjecture', generalizing the analogous conjecture of Bass on
$K$-groups, is that $\HMXr q$ should be finitely generated.
Except for the case of $\dim(X)=1$ where this is known for all $q$ and $r$ (Quillen),
the only other general case where the finite generation is is known is
$\HMXd {2d} =\CH^d(X)=\CH_0(X)$ where $d=\dim(X)$, which is
a consequence of higher dimensional class field theory
(\cite{B3}, \cite{KS1} and \cite{CTSS}).

One way to approach the problem is to look at
an \'etale cycle map constructed by Geisser and Levine \cite{GL2} :
\begin{equation}\label{ccmap}
\rho^{r,q}_X \;:\; \CH^r(X,q;\nz) \to H^{2r-q}_{\et}(X,\nz(r)),
\end{equation}
Here
$$
CH^r(X,q; \nz) = H_{q}(z^r(X,\bullet)\otimes^{\bL} \nz)\,,
$$
is the higher Chow group with finite coefficients, which fits
into a short exact sequence:
$$0\to CH^r(X,q)/n \to CH^r(X,q; \nz) \to CH^r(X,q-1)[n] \to 0,$$
and $\nz(r)$ is the complex of \'etale sheaves on $X$:
$$
\nz(r)=\mu_m^{\otimes r} \oplus \DWnulog r X[-r],
$$
if $n=mp^r$ and $(p,m)=1$ with $p=\ch(F)$
(cf. \eqref{def.nz} and \eqref{Tatetwist}). Using finiteness results on
\'etale cohomology, the injectivity of
$\rho^{q,r}_X$ would imply a result which relates to the folklore
conjecture like the weak Mordell-Weil theorem relates to the strong one.
\medbreak

In case $r>d:=\dim(X)$ it is easily shown that $\rho_X^{r,q}$ is
an isomorphism assuming the Bloch-Kato conjecture (see \ref{cyclekh}).
An interesting phenomenon emerges for $\rho_X^{r,q}$ with $r=d$.
The Bloch-Kato conjecture implies that there is a long exact sequence
(see \ref{cyclekh}):
\begin{multline}\label{les1}
\cdots\to
\KHetnz {q+2} X \to \CH^d(X,q;\nz) @>{\rho_X^{d,2d-q}}>>
H^{2d-q}_{\et}(X,\nz(d)) \\
\to \KHetpnz {q+1} X \to\cdots
\end{multline}

Hence Theorem \ref{mainthmintro1} implies the following:

\begin{thm}\label{cor1intro}
Let $X$ be smooth projective of pure dimension $d$ over a finite field $F$.
Let $t,n\geq 1$ be integers.
Assume $\BKlz {t+2}$ for all primes $l|n$.
Assume further either $t\leq 2$ or {$\bf(RS)_{d}$} or {$\bf(RES)_{t}$}.
Then
$$
\rho^{d,q}_X \;:\; \CH^d(X,q,\nz) \isom H^{2d-q}_{\et}(X,\nz(d))
\qforall q\leq t.
$$
In particular $ \CH^d(X,q,\nz)$ is finite under the assumption.
\end{thm}

Using results of \cite{Kah}, \cite{Ge1} and \cite{J2} generalizing a seminal
result of Soul\'e \cite{So}, we deduce from \ref{cor1intro} the following:

\begin{cor}\label{cor2intro}
Let the assumption be as in \ref{cor1intro}.
Assume further that $X$ is finite-dimensional in the sense of
Kimura \cite{Ki} and O'Sullivan \cite{OSu}
(which holds if $X$ is a product of abelian varieties and curves).
Then there is an isomorphism of finite groups
$$
\CH^d(X,q) \simeq \underset{\text{all prime $l$}}{\bigoplus}
H^{2d-q}_{\et}(X,\zl(d))\qforall 1 \leq q\leq t.
$$
\end{cor}
\bigskip

Finally we explain briefly the strategy to prove Theorem \ref{mainthmintro2}.
The first key observation is that the conclusion of \ref{mainthmintro2}
implies the following fact: For $X$, projective smooth over $B=\Spec(k)$, and
for a simple normal crossing divisor $Y$ on $X$, the Kato homology
$\KH a U$ of $U=X-Y$ has a combinatoric description
as the homology of the complex
$$
(\Lambda)^{\pi_0(Y^{(d)})} \rightarrow (\Lambda)^{\pi_0(Y^{(d-1)})}
\rightarrow \ldots \rightarrow  (\Lambda)^{\pi_0(Y^{(1)})}
\rightarrow(\Lambda)^{\pi_0(X)}\,,
$$
where $\Lambda=H_0(B)$ and $\pi_0(Y^{(a)})$ is the set of
the connected components of the sum of
all $a$-fold intersections of the irreducible components
$Y_1,\dots,Y_N$ of $Y$.
Conversely the vanishing of the Kato homology of $X$ is deduced from such
a combinatoric description of $\KH a U$ for a suitable choice of
$U=X-Y$.
\medbreak

On the other hand, the conditions $(H1)$ through $(H3)$ of \ref{mainthmintro2}
imply that if one of the divisors $Y_1,\dots, Y_N$ on $X$ is very ample,
$H_a(U)$ for $a \leq d$ has the same combinatoric description.
Recalling that the spectral sequence \eqref{spectralsequence1}
$$
E^1_{a,b}(U)=\bigoplus_{x\in U_a}H_{a+b}(x)~~\Rightarrow ~~H_{a+b}(U)
$$
satisfies
$E^1_{a,b}(U) =0$ for $b<0$ and $E^2_{a,b}(U) = \KH a U$ for $b=0$,
the desired combinatoric description of $\KH a U$
is then deduced from the following vanishing:
\begin{equation}\label{vanishingZB}
\ZBU a b \infty b:=\ZU a b \infty / \BU a b {b} \;= 0
\qfor b\geq 1.
\end{equation}
Here
$$
E^1_{a,b}(U)=\ZU ab 0\supset\ZU ab r\supset \ZU ab \infty\supset
\BU ab \infty\supset  \BU ab r \supset \BU ab 0=0,
$$
is the standard notation for the spectral sequence so that
$$
\EU ab {r+1}=\ZU ab r/ \BU ab r,\quad
\ZU ab \infty=\underset{r\geq 0}{\cap}\ZU ab r,\quad
\BU ab \infty=\underset{r\geq 0}{\cup}\BU ab r.
$$
In order to show the vanishing, we pick up any element
$$\alpha\in \ZBU a b \infty b$$
and then take a hypersurface section of high degree
$Z\subset X$ containing the support $Supp(\alpha)$ of $\alpha$
so that $\alpha$ is killed under the restriction
$$
\ZBU a b \infty b \to \ZBUZ a b \infty b.
$$
The point is that the assumption $\RES {q}$ allows us to
make a very careful choice of $Z$, after desingularizing $Supp(\alpha)$,
so as to ensure by the induction on $\dim(U)$ the injectivity of
the restriction map, which implies $\alpha=0$.
The last step of the argument hinges on a general lemma proved in \S1
concerning the exactness of the following sequence:
$$
\ZBZcapU a b \infty b \to \ZBU a b \infty b \to \ZBUZ a b \infty b.
$$
Finally we note that taking $H = H^{\et}(-,\nz)[-1]$, the vanishing
\eqref{vanishingZB} may be viewed as an analog of the weak Lefschetz theorem
for cycle modules in the sense of Rost \cite{R}.
This will be explained explicitly for terms in lower degrees
in Corollary \ref{cormainthmIfinite} in \S5.

\bigskip
We note that the Kato conjecture for varieties over finite fields is studied
also by a different method in a paper \cite{J1} by the first author.
The method introduced in this paper was found by the second author
independently. It has been applied in \cite{SS} to study cycle class map
for 1-cycles on arithmetic schemes over the ring of integers in a local
field to provide new finiteness results.

It is not difficult to extend the method of this paper to study the Kato
conjecture and motivic cohomology of arithmetic schemes over the ring of
integers in a local field,
at least restricted to the prime-to-$p$ part, where $p$ is the residue
characteristic of the local field. In order to deal with the $p$-part and
the case of arithmetic schemes over the ring of integers in a number field,
one need develop a new input from $p$-adic Hodge theory.
This is a work in progress [JS3].
\bigskip

The authors thank Prof. T. Geisser
for helpful comments. The second author is grateful to the first author
for several opportunities to stay in the department of Mathematics at
University of Regensburg where he enjoyed warm hospitality.

\bigskip

\section{Fundamental Lemma}\label{FL}
\bigskip

Throughout this paper we fix a regular connected Noetherian base scheme $B$ and
work with a category $\cC$ of separated schemes of finite type over $B$ such
that for any object $X$ in $\cC$, every closed immersion $i:Y\hookrightarrow X$
and every open immersion $j:V\hookrightarrow X$ is (a morphism) in $\cC$.
For $X\in Ob(\cC)$ we define $\dim(X)$ to be the Krull dimension of any
compactification $\overline{X}$ of $X$ over $B$ (i.e., $\overline{X}$ is
proper over $B$ and there is an open immersion $X \hookrightarrow \overline{X}$
of $b$-schemes). This does not depend on the choice of compactification.
For an integer $a\geq 0$ let $X_a$ denotes the set of such $x\in X$ that
$\dim(\overline{\{x\}})=a$. Then one can check:
\begin{equation}\label{dimension}
X_a\cap Y =Y_a\quad\text{ for $Y$ locally closed in $X$}.
\end{equation}

Let $\Mod$ be the category of modules.

\begin{defn}
(a) Let $\cC_\ast$ be the category with the same objects as
$\cC$, but where morphisms are just the proper maps in
$\cC$. A homology theory $H=\{H_a\}_{a\in \bZ}$ on $\cC$ is
a sequence of covariant functors:
$$
H_a(-):~\cC_*\rightarrow \Mod
$$
satisfying the following conditions:
\begin{itemize}
\item [(i)] For each open immersion $j:V\hookrightarrow X$ in
$\cC$, there is a map $j^*:H_a(X)\rightarrow H_a(V)$,
associated to $j$ in a functorial way.
\item [(ii)] If $i:Y\hookrightarrow X$ is a closed immersion in $X$,
with open complement $j:V\hookrightarrow X$, there is a long exact sequence
(called localization sequence)
$$
\cdots\rmapo{\partial} H_a(Y) \rmapo{i_*} H_a(X) \rmapo{j^*} H_a(V) \rmapo{\partial} H_{a-1}(Y)
\longrightarrow \cdots.
$$
(The maps $\partial$ are called the connecting morphisms.) This
sequence is functorial with respect to proper maps or open
immersions, in an obvious way.
\end{itemize}

(b) A morphism between homology theories $H$ and $H'$ is a
morphism $\phi: H \rightarrow H'$ of functors on $\mathcal
C_\ast$, which is compatible with the long exact sequences from (ii).
\end{defn}
\medbreak

We give basic examples.

\begin{ex}\label{exH1}
Let $\cK$ be a bounded complex of \'etale sheaves of torsion modules on $B$.
Then one gets a homology theory $H = H^{\et}(-,\cK)$ on $\cC$ by defining
$$
H^\et_a(X,\cK):=H^{-a}(X_{\acute{e}t}, R\,f^{!}\cK),
\qfor f:X\rightarrow B \text{ in }\cC
$$
called the \'etale homology of $X$ over $B$ with values in $\cK$.
Here $Rf^!$ is the right adjoint of $Rf_!$ defined in [SGA 4], XVIII, 3.1.4.
\end{ex}
\bigskip

\begin{ex}\label{exH2}
Let $\cK$ be as in \ref{exH1}. One defines a homology theory
$H^D(-,\cK)$ on $\cC$ by:
$$
\HDK a X:=\Hom\big(H^{a}(B, R\,f_{!}f^*\cK),\qz\big),
\qfor f:X\rightarrow B \text{ in }\cC.
$$
\end{ex}
\bigskip

We fix a homology theory $H$ on $\cC$.
For every $X \in Ob(\cC)$, we have the spectral sequence of homological type,
called the niveau spectral sequence:
\begin{equation}\label{spectralsequence1}
E^1_{p,q}(X)=\bigoplus_{x\in X_p}H_{p+q}(x)~~\Rightarrow ~~H_{p+q}(X) \qwith
H_a(x)=\lim{V\subseteq \overline{\{x\}}} H_a(V).
\end{equation}
Here the limit is over all open non-empty subschemes
$V\subseteq \overline{\{x\}}$. This spectral sequence
is covariant with respect to proper morphisms in $\cC$ and
contravariant with respect to open immersions.
We briefly recall the construction of this spectral sequence
given by Bloch-Ogus \cite{BO}.
For $T\in \cC$ let $\cZ_p(T)$ be the set of closed subsets $Z~\subset~T$ of
dimension $\leq p$, ordered by inclusion, and let $\cZ_p/\cZ_{p-1}(T)$ be
the set of pairs $(Z,Z')\in \cZ_p\times \cZ_{p-1}$ with $Z'~\subset~Z$,
again ordered by inclusion. For every $(Z,Z') \in \cZ_p/\cZ_{p-1}(X)$,
one has the exact localization sequence
$$
\ldots \rightarrow H_a(Z') \rightarrow H_a(Z) \rightarrow
H_a(Z\setminus Z')
 \mathop{\rightarrow}\limits^{\partial} H_{a-1}(Z') \rightarrow \ldots,
$$
Taking its limit over $\cZ_p/\cZ_{p-1}(X)$, we get the exact sequence
\begin{equation}\label{exactsequence1}
\ldots H_a(\cZ_{p-1}(X)) \rightarrow H_a(\cZ_p(X)) \rightarrow
H_a(\cZ_p/\cZ_{p-1}(X)) \mathop{\rightarrow}\limits^{\delta}
H_{a-1}(\cZ_{p-1}(X)) \ldots.
\end{equation}
The collection of these sequences, together with the
fact that one has $H_\ast(\cZ_p(X)) = 0$ for $p<0$ and
$H_\ast(\cZ_p(X)) = H_\ast(X)$ for $p\geq \dim X$, gives the
spectral sequence in a standard way, e.g., by exact couples. Here
\begin{equation}\label{spectralsequence2}
E^1_{p,q}(X) = H_{p+q}(\cZ_p/\cZ_{p-1}(X)) = \bigoplus_{x\in X_p}
H_{p+q}(x).
\end{equation}

The $E^1$-differentials are the compositions
$$
H_{p+q}(\cZ_p/\cZ_{p-1}(X)) \rmapo{\delta}
H_{p+q-1}(\cZ_{p-1}(X)) \rightarrow
H_{p+q-1}(\cZ_{p-1}/\cZ_{p-2}(X)).
$$
The $E^r$-differentials are denoted by:
$$
\dX pq r \;:\; \EX pq r \to \EX {p-r}{q+r-1} r.
$$
We will use the standard notation:
\begin{equation}\label{terms}
E^1_{p,q}(X)=\ZX pq 0\supset\ZX pq r\supset \ZX pq \infty\supset
\BX pq \infty\supset  \BX pq r \supset \BX pq 0=0,
\end{equation}
where
$$
\EX pq {r+1}=\ZX pq r/ \BX pq r,\quad
\ZX pq \infty=\underset{r\geq 0}{\cap}\ZX pq r,\quad
\BX pq \infty=\underset{r\geq 0}{\cup}\BX pq r.
$$
We also denote
$$
\ZBX pq rs =\ZX pq r/ \BX pq s.
$$
\bigskip

In what follows we fix $X\in Ob(\cC)$,
a closed subscheme $i: Y \hookrightarrow X$ with
$j: V = X \setminus Y \hookrightarrow X$, the open complement.
The property \eqref{dimension} allows us to have the following maps of
the spectral sequences (cf. \cite{JS1}, Prop.2.9)
\begin{equation}\label{mapofss}
i_*: E^1_{p,q}(Y) \to E^1_{p,q}(X),\quad
j^*: E^1_{p,q}(X) \to E^1_{p,q}(V),\quad
\partial : E^2_{p,q}(V)^{(-)} \to E^2_{p-1,q}(Y),
\end{equation}
where the superscript $\null^{(-)}$ means that all differentials
in the original spectral sequence are multiplied by $-1$.
We have the short exact sequence
\begin{equation}\label{E1exactsequence}
0 \rightarrow E^1_{p,q}(Y) \rmapo{i_\ast} E^1_{p,q}(X)
\rmapo{j^\ast} E^1_{p,q}(V) \rightarrow 0
\end{equation}
and the long exact sequence
$$
\ldots\rightarrow E^2_{p,q}(Y) \rmapo{i_\ast} E^2_{p,q}(X)
\rmapo{j^\ast} E^2_{p,q}(V) \rmapo{\partial}
E^2_{p-1,q}(Y) \rightarrow \ldots.
$$
For $r\geq 3$ the sequence
$$
E^r_{p,q}(Y) \rmapo{i_\ast} E^r_{p,q}(X)
\rmapo{j^\ast} E^r_{p,q}(V)
$$
is not anymore exact in general. The following result will play a crucial role
in the proof of the main results of this paper.

\begin{thm}\label{FL}
Fix integers $p,q\geq 0$. Assume that there is an integer $e\geq 1-q$ such that
$$
\ZBV {p-a}{q+a}{\infty}{q+a+e} =0\qforall a\geq 1,
$$
$$
\ZBY {p-a}{q+a}{\infty}{q+a+e} =0 \qforall a \mbox{ with }  -(q+e-1) \leq  a\leq -1.
$$
Then the following sequence is exact:
$$
\ZBY pq \infty {q+e} \rmapo {i_*} \ZBX pq \infty {q+e}
\rmapo {j^*} \ZBV pq \infty {q+e}.
$$
\end{thm}

\bigskip

We need some preliminaries for the proof of the main theorem.
Recall that we have the exact sequence
\begin{equation}\label{les1}
\HXZ {p-1} \to \HXZ p \rmapo{\piX pq} \EX pq 1 \rmapo{\delX pq} \HXZa {p-1}
\end{equation}
For each integer $r\geq 0$ we put
\begin{equation}\label{Kpqr}
\KX pq r =\Ker\big( \HXZ p \to \HXZ {p+r}\big).
\end{equation}
Then we have
\begin{equation}\label{Bpqr}
\BX pq r=\piX pq(\KX pq r),
\end{equation}
\begin{equation}\label{Zpqinfty}
\ZX pq \infty= \Image(\piX pq)=\Ker(\delX pq),
\end{equation}
\begin{equation}\label{Zpqr}
\displaystyle{
\ZX pq r= {\delX pq}^{-1}(\frac{\HXZa {p-1-r}}{\KX {p-1-r}{q+r} r}),}
\end{equation}
and $\dX pq {r+1}:\EX pq {r+1}\to \EX {p-1-r}{q+r} {r+1}$ is induced by
$$ \displaystyle{
\ZX pq r \rmapo{\delX pq} \frac{\HXZa {p-1-r}}{\KX {p-1-r}{q+r} r}
\rmapo{\piX {p-1-r}{q+r}} \ZBX {p-1-r}{q+r} 0 r.
}$$
\bigskip

We now introduce an object that plays a key role in the proof of \ref{FL}.
\begin{defn}\label{Fpqr}
We set
$$
\displaystyle{
\FF pq r =\Ker\big(\HXZ {p+r} \rmapo{j^*}
\frac{\HVZ {p+r}}{\Image(\HVZ p)}\big).}
$$
\end{defn}

Note
\begin{equation*}\label{note1}
\Image\big(\HYZ {p+r} @>{i_*}>> \HXZ {p+r}\big) \subset \FF pq r.
\end{equation*}
\begin{equation*}\label{note2}
\Image\big(\HXZ {p} @>>> \HXZ {p+r}\big) \subset \FF pq r.
\end{equation*}
By definition there is a natural map
\begin{equation}\label{gFpqr}
\displaystyle{
\gF pq r \;:\; \FF pq r \rmapo{j^*}
\frac{\HVZ p}{\KV p q r} \rmapo{\piV p q} \ZBV p q \infty r}.
\end{equation}
Noting
$\Ker(\piV p q : \HVZ p \to \EV pq 1)=\Image(\HVZ {p-1})$, we have
\begin{equation}\label{gFpqrler}
\Ker(\gF pq r) =\FF {p-1}{q+1}{r+1}.
\end{equation}
There is a natural map
\begin{equation}\label{psiFpqr}
\psiF pq r\;:\; \ZV {p+r+1}{q-r} r \to \FF pq r/i_*\KY {p+r}{q-r} 1,
\end{equation}
where $i_*: \HYZ {p+r} \to \HXZ {p+r}$.
Indeed consider the composite map
$$
\displaystyle{
\ZV {p+r+1}{q-r} r \hookrightarrow \EV {p+r+1}{q-r} 1
\underset{(j^*)^{-1}}{\isom} \frac{\EX {p+r+1}{q-r} 1}{\EY {p+r+1}{q-r} 1}
\rmapo{\delX {p+r+1}{q-r}} \frac{\HXZ {p+r}}{i_*\KY {p+r}{q-r} 1}
}$$
where we note
$\Image(\EY {p+r+1}{q-r} 1 \to \HYZ {p+r})=\KY {p+r}{q-r} 1$.
By \eqref{Zpqr} its image lies in $\FF pq r/i_*\KY {p+r}{q-r} 1$ and
the following sequence is exact:
\begin{equation}\label{psicomplex}
\ZV {p+r+1}{q-r} r \rmapo{\psiF pq r} \FF pq r/i_*\KY {p+r}{q-r} 1
\rmapo{\tau} \FF pq {r+1},
\end{equation}
where $\tau$ is induced by the natural map
$\HXZ {p+r} \to \HXZ {p+r+1}$.
The following diagram is commutative and all sequences are exact:
\begin{equation}\label{gFpqrcomm}
\begin{CD}
 \ZV {p+r+1}{q-r}{r+1}  @>{\hookrightarrow}>>  \ZV {p+r+1}{q-r}{r}
@>{\dV {p+r+1}{q-r}{r+1}}>> \BBV pq {r+1}r \\
 @VV{\psiF {p-1}{q+1}{r+1}}V @VV{\psiF {p}{q}{r}}V @VVV \\
 \FF {p-1}{q+1}{r+1}/i_*\KY {p+r}{q-r} 1 @>{\hookrightarrow}>>
\FF pq r/i_*\KY {p+r}{q-r} 1 @>{\gF pq r}>> \ZBV pq \infty r \\
 @VV{\tau}V @VV{\tau}V @VVV \\
 \FF {p-1}{q+1}{r+2} @>{\hookrightarrow}>> \FF pq {r+1}
@>{\gF pq {r+1}}>> \ZBV pq \infty {r+1} \\
\end{CD}
\end{equation}

\begin{lem}\label{fl-lem1}
Assume $r\geq 1$.
\begin{itemize}
\item[(1)]
There exists a unique map
$$ \fF pq r \;:\; \FF pq r \to \ZY {p+r}{q-r} r$$
whose composition with
$\ZY {p+r}{q-r} r\hookrightarrow \EY {p+r}{q-r} 1 @>{i_*}>> \EX {p+r}{q-r} 1$
is
$$\FF pq r \hookrightarrow \HXZ {p+r} \rmapo{\piX {p+r}{q-r}}
 \EX {p+r}{q-r} 1.$$
\item[(2)]
The composite map
$$
\HYZ {p+r} @>{i_*}>> \FF pq r @>{\fF pq r }>> \ZY {p+r}{q-r} r
\hookrightarrow \EY {p+r}{q-r} 1
$$
is the natural map
$\HYZ {p+r} \to \HYZZ {p+r}.$
\item[(3)]
The following diagram is commutative up to $\pm 1$.
$$
\begin{CD}
\ZBV pq \infty r @>{\partial}>> \ZBY {p-1}q \infty r\\
@AA{\gF pq r}A                 @AA{\dY {p+r}{q-r}{r+1}}A\\
\FF pq r @>{\fF pq r}>> \ZY {p+r}{q-r} r \\
\end{CD}
$$
\item[(4)]
\eqref{psicomplex} (with $r$ replaced by $r$-$1$) extends to the following exact sequence:
$$
\displaystyle{
\ZV {p+r}{q-r+1} {r-1} \rmapo{\psiF pq {r-1}}
\frac{\FF pq {r-1}}{i_*\KY {p+r-1}{q-r+1} 1} \rmapo{\tau} \FF pq {r}
\rmapo{\fF pq {r}} \ZY {p+r}{q-r} {r}.
}$$
\end{itemize}
\end{lem}
\bigskip

The proof of \ref{fl-lem1} will be given later in this section.

\begin{thm}\label{FLa}
Let the assumption be as in \ref{FL}.
Let $x\in \HXZ p$ and assume $j^*(\piX pq(x))\in \BV pq {q+e}$.
Then we have
$x\in \KX pq {q+e} + \FF {-1}{p+q+1} {p+1}.$
\end{thm}
\medbreak

We first deduce \ref{FL} from \ref{FLa}.
By \eqref{Bpqr} it suffices to show that the conclusion of \ref{FLa} implies
$$x\in \KX pq {q+e} + \Image(\HXZ {p-1}) +
\Image(\HYZ p \rmapo{i_*} \HXZ p).$$
This follows from the commutative diagram
$$
\begin{CD}
\HYZ p @>{\piY pq}>> \ZBY pq \infty {q+e}\\
@VV{i_*}V @|\\
\FF {-1}{p+q+1}{p+1} @>{\fF {-1}{p+q+1}{p+1}}>> \ZBY pq {p+1}{q+e}\\
\end{CD}
$$
together with the fact that
$\piY pq$ is surjective by \eqref{Zpqinfty} and that
$\Ker(\fF {-1}{p+q+1}{p+1})$ lies in the image of
$\FF {-1}{p+q+1}{p}\subset \HXZ {p-1}$
due to \ref{fl-lem1}(4) and \ref{fl-lem1}(4)(2) and \eqref{Bpqr}.
$\square$
\bigskip

Let $e$ be the integer in \ref{FL} and set
$$
\tFF pq r =\FF pq r/i_*\KY {p+r}{q-r} {q-r+e}.
$$
For any $z\in \HXZ p$ and any integer $0\leq t\leq q+e-1$, let
$$z^{(t)}\in \tFF pq t\subset \HXZ {p+t}/i_*\KY {p+t}{q-t}{q-t+e}$$
be the image of $z$ under the map induced by $\HXZ {p} \rightarrow \HXZ {p+t}$ (note $q-t+e\geq 1$).
By induction \ref{FLa} is deduced from the following claims.

\begin{claim}\label{fl-claim1}
Let $x\in \HXZ p$ and assume $j^*(\piX pq (x))\in \BV pq {q+e}$.
Assume:
$$\ZBV {p-t}{q+t}\infty {q+t+e}=0 \qforall 1\leq t\leq p.
\leqno(*1)$$
$$\ZBY {p+t}{q-t}\infty {q-t+e}=0 \qforall 1\leq t\leq q+e-1.
\leqno(*2)$$
Then there exists $u\in \KX pq {q+e}$ such that
$(x-u)^{(q+e-1)}\in \tFF {-1}{p+q+1}{p+q+e}\subset \tFF {p}{q}{q+e-1}$.
\end{claim}

\begin{claim}\label{fl-claim2}
Fix an integer $0\leq r\leq q+e-2$.
Let $x\in \HXZ p$ and assume
$x^{(r+1)}\in \tFF {-1}{p+q+1} {p+r+2}\subset \tFF pq {r+1}$.
Assume
$$\ZBY {p+t}{q-t}\infty {q-t+e}=0 \qforall 1\leq t\leq r.
\leqno(*3)$$
Then there exists $u\in \KX pq {q+e}$ such that
$(x-u)^{(r)} \in \tFF {-1}{p+q+1} {p+r+1}\subset \tFF pq r$.
\end{claim}
\medbreak

For the proof of the claims, we need the following lemmas.

\begin{lem}\label{fl-lem2}
If $\ZBY {p+t}{q-t}\infty {q-t+e}=0$, then
$\tFF {-1}{p+q+1}{p+t}\to \tFF {-1}{p+q+1}{p+t+1}$
is surjective.
\end{lem}
\begin{proof}
This follows from the exact sequence
$$ \tFF {-1}{p+q+1}{p+t}\to \tFF {-1}{p+q+1}{p+t+1}
\rmapo{\fF {-1}{p+q+1}{p+t+1}} \ZBY {p+t}{q-t}{p+t+1}{q-t+e}$$
which is deduced from \ref{fl-lem1}(4) and \eqref{Bpqr},
together with the facts that
$$
\fF {-1}{p+q+1}{p+t+1}(i_*\KY {p+t}{q-t}{q-t+e})=\BY {p+t}{q-t}{q-t+e}
$$
by \ref{fl-lem1}(2) and \eqref{Bpqr} and that
$\ZY {p+t}{q-t}{p+t+1}=\ZY {p+t}{q-t}{\infty}$.
\end{proof}

\begin{lem}\label{fl-lem3}
Consider the maps
$$
\tpsiF pq r : \ZV {p+r+1}{q-r} r \to \tFF pq r, \quad
\iota : \HXZ p \to \tFF pq r,
$$
where the first map is induced by $\psiF pq r$ \eqref{psiFpqr} and
the second by the natural map $\HXZ p\to \HXZ {p+r}$.
Assuming $r\leq q+e-1$, we have
$$ \Image(\tpsiF pq r)\cap \Image(\iota)\subset \iota(\KX pq {q+e}).$$
\end{lem}
\begin{proof}
We have the commutative diagram
$$
\begin{CD}
@. \HXZ p @>{\alpha}>> \HXZ {p+q+e} \\
@. @VV{\iota}V               @AA{\beta}A \\
\ZV {p+r+1}{q-r} r @>{\tpsiF pq r}>> \tFF pq r @>\subset>>
\HXZ {p+r}/i_*\KY {p+r}{q-r}{q-r+e}\\
\end{CD}
$$
where $\beta$ exists since
$\KY {p+r}{q-r}{q-r+e}=
\Ker\big(\HYZ {p+r}\to \HYZ {p+q+e}\big)$.
By the assumption we have $p+r+1\leq p+q+e$ so that \eqref{psicomplex} implies
$\Image(\tpsiF pq r)\subset\Ker(\beta)$.
Therefore \ref{fl-lem2} follows by noting that
$\Ker(\alpha)=\KX pq {q+e}$.
\end{proof}
\bigskip

Now we show \ref{fl-claim1}. We use the following commutative diagram
with exact horizontal sequences, which is deduced from \eqref{gFpqrcomm}:
\begin{equation*}\label{gFpqrcomm1}
\begin{CD}
 \ZV {p+q+e}{-e+1}{q+e+s} @>{\hookrightarrow}>> \ZV {p+q+e}{-e+1}{q+e+s-1}
@>{\dV {p+q+e}{-e+1}{q+e+s}}>> \BBV {p-s}{q+s} {q+e+s} {q+e+s-1} \\
 @VV{\psiF {p-s-1}{q+s+1}{q+e+s}}V @VV{\psiF {p-s}{q+s}{q+e+s-1}}V @VVV \\
 \tFF {p-s-1}{q+s+1}{q+e+s} @>{\hookrightarrow}>> \tFF {p-s}{q+s} {q+e+s-1}
@>{\gF {p-s}{q+s} {q+e+s-1}}>> \ZBV {p-s}{q+s} \infty {q+e+s-1} \\
\end{CD}
\end{equation*}
Take $y_1\in \ZV {p+q+e}{-e+1}{q+e-1}$ such that
$j^*(\piX pq (x))=\dV {p+q+e}{-e+1}{q+e}(y_1)$.
By the above diagram with $s=0$ we have
$$\epsilon_1:=x^{(q+e-1)}-\psiF pq {q+e-1}(y_1)\in
\tFF {p-1}{q+1} {q+e}\subset \tFF {p}{q} {q+e-1}.$$
Let $s\geq 1$ be an integer and assume that there exists
$z_s\in \ZV {p+q+e}{-e+1}{q+e-1}$ such that
$$\epsilon_s:=x^{(q+e-1)}-\psiF pq {q+e-1}(z_s)\in
\tFF {p-s}{q+s} {q+e+s-1}\subset \tFF {p}{q} {q+e-1}.$$
By the assumption $(*1)$ we can find
$y_s\in \ZV {p+q+e}{-e+1}{q+e+s-1}$ such that
$\gF {p-s}{q+s}{q+e+s-1}(\epsilon_s)=\dV {p+q+e}{-e+1}{q+e+s}(y_s)$.
By the above diagram we get
$$\epsilon_{s+1}:=\epsilon_s-\psiF {p-s}{q+s} {q+e+s-1}(y_s)\in
\tFF {p-s-1}{q+s+1} {q+e+s}\subset \tFF {p}{q} {q+e-1}.$$
By induction this shows that there exists
$z\in \ZV {p+q+e}{-e+1}{q+e-1}$ such that
$$\epsilon:=x^{(q+e-1)}-\psiF pq {q+e-1}(z)\in
\tFF {-1}{p+q+1} {p+q+e}\subset \tFF {p}{q} {q+e-1}.$$
By the assumption $(*2)$ \ref{fl-lem2} implies
$\tFF {-1}{p+q+1} {p+1} \to \tFF {-1}{p+q+1} {p+q+e}$
is surjective. Since $\tFF {-1}{p+q+1} {p+1}\subset\HXZ p/i_*\KY pq {q+e}$
we get $\epsilon\in \tFF {-1}{p+q+1} {p+q+e}\cap\Image(\HXZ p)$
so that $\psiF pq {q+e-1}(z)\in \Image(\HXZ p)$.
By \ref{fl-lem3}, $\psiF pq {q+e-1}(z)\in \Image(\KX pq {q+e})$,
which completes the proof of \ref{fl-claim1}.
\medbreak

Next we show \ref{fl-claim2}.
We use the following commutative diagram deduced from \eqref{gFpqrcomm}:
\begin{equation*}\label{gFpqrcomm2}
\begin{CD}
 \ZV {p+r+1}{q-r}{r+s+1}  @>{\hookrightarrow}>>  \ZV {p+r+1}{q-r}{r+s}
@>{\dV {p+r+1}{q-r}{r+s+1}}>> \BBV {p-s}{q+s} {r+s+1}{r+s} @>>> 0\\
 @VV{\psiF {p-s-1}{q+s+1}{r+s+1}}V @VV{\psiF {p-s}{q+s}{r+s}}V @VVV \\
 \tFF {p-s-1}{q+s+1}{r+s+1} @>{\hookrightarrow}>> \tFF {p-s}{q+s}{r+s}
 @>{\gF {p-s}{q+s} {r+s}}>> \ZBV {p-s}{q+s} \infty {r+s} \\
 @VV{\tau}V @VV{\tau}V @VVV \\
 \tFF {p-s-1}{q+s+1}{r+s+2} @>{\hookrightarrow}>> \tFF {p-s}{q+s}{r+s+1}
@>{\gF {p-s}{q+s} {r+s+1}}>> \ZBV {p-s}{q+s} \infty {r+s+1} \\
\end{CD}
\end{equation*}
where the vertical and horizontal sequences are exact.
By the diagram with $s=0$ the assumption
$x^{(r+1)}\in \tFF {-1}{p+q+1}{p+r+2}$ implies that there exists
$y_1\in \ZV {p+r+1}{q-r}{r}$ such that
$\gF pq r(x^{(r)})=\dV {p+r+1}{q-r}{r+1}(y_1)$ and hence
$$\epsilon_1:=x^{(r)}-\psiF pq r(y_1)\in
\tFF {p-1}{q+1} {r+1}\subset \tFF {p}{q} {r}.$$
Let $p\geq s\geq 1$ be an integer and assume that there exists
$z_s\in \ZV {p+r+1}{q-r} r$ such that
$$\epsilon_s:=x^{(r)}-\psiF pq r(z_s)\in
\tFF {p-s}{q+s} {r+s}\subset \tFF {p}{q} {r}.$$
By \eqref{psicomplex} the assumption $x^{(r+1)}\in \tFF {-1}{p+q+1}{p+r+2}$
implies the image of $\epsilon_s$ in $\tFF {p-s}{q+s}{r+s+1}$ lies in
$\tFF {-1}{p+q+1}{p+r+2}$.
Hence the same argument as before shows that there exists
$y_s\in \ZV {p+r+1}{q-r}{r+s}$ such that
$$\epsilon_{s+1}:=\epsilon_s-\psiF {p-s}{q+s} {r+s}(y_s)\in
\tFF {p-s-1}{q+s+1} {r+s+1}\subset \tFF {p}{q} {r}.$$
By induction this shows that there exists
$z\in \ZV {p+r+1}{q-r} r$ such that
$$\epsilon:=x^{(r)}-\psiF pq r(z)\in \tFF {-1}{p+q+1} {p+r+1}.$$
By the assumption $(*3)$ \ref{fl-lem2} implies
$\tFF {-1}{p+q+1} {p+1} \to \tFF {-1}{p+q+1} {p+r+1}$
is surjective. Now \ref{fl-claim2} follows from \ref{fl-lem3}
by the same argument as before.

\bigskip

\noindent
\it Proof of \ref{fl-lem1}. \rm

First we show (1).
The uniqueness of $\fF pq r$ is a direct consequence of the injectivity
of $\EY pq 1 \to \EX pq 1$. To show its existence, we consider the
following commutative diagram:
$$
\begin{CD}
@.  \FF pq r @>{\subset}>> \HXZ {p+r} @>{j^*}>> \HVZ {p+r} @.\\
@.  @. @VV{\piX {p+r}{q-r}(X)}V @VV{\piV {p+r}{q-r}(V)}V  @.\\
0 @>>> \EY {p+r}{q-r} 1 @>{i_*}>> \EX {p+r}{q-r} 1 @>{j^*}>>
\EV {p+r}{q-r} 1 @>>> 0 \\
\end{CD}
$$
We have $j^*\circ \piX {p+r}{q-r}(X)(\FF pq r)=0$ since
$\Ker(\piV {p+r}{q-r} (V))$ contains $\Image(\HVZ p)$ for $r\geq 1$
(cf. \eqref{les1}). Hence we get the induced map
$\fF pq r :\FF pq r \to \EY {p+r}{q-r} 1$.
It remains to show that its image lies in $\ZY {p+r}{q-r} r$.
We consider the following diagram:
\begin{equation}\label{gfcomm}
\begin{CD}
 \FF pq r @>{j^*}>> \HVZ {p+r} @<<< \HVZ {p} \\
@VV{\fF pq r}V @VV{\partial}V  @VV{\partial}V \\
 \EY {p+r}{q-r} 1 @>{\delY {p+r}{q-r}}>> \HYZa {p+r-1} @<<< \HYZa {p-1}\\
\end{CD}
\end{equation}
where $\partial$ is the map inducing $\partial$ in \eqref{mapofss}.
Noting
$\ZY {p+r}{q-r} r={\delY {p+r}{q-r}^{-1}}(\Image(\HYZ {p-1})$
(cf. \ref{Zpqr}), it remains to show that the squares are commutative.
For this we need to recall the definition of $\partial$.
For a closed subset $T\subset V$ let $\bar{T}$ be the closure.
For $T\in \cZ_s(V)$ one then has $\bar{T}\cap Y \in \cZ_{s-1}(Y)$,
and we have the localization sequences
$$
\cdots\to H_a(\bar{T}\cap Y) \to H_a(\bar{T}) \to H_a(T) \rmapo{\partial}
H_{a-1}(\bar{T}\cap Y) \to \cdots$$
Taking the limits over $T\in \cZ_r(V)$, one gets
$$ \partial : H_a(\cZ_r(V)) \to H_{a-1}(\cZ_{r-1}(Y)).$$
From the definition, the commutativity of the right square of \eqref{gfcomm}
is obvious. For the left square, it suffices to check
that the following diagram is commutative:
$$
\begin{CD}
H_{s+t}(\cZ_s(X)) @>{\piX st}>> \EX st 1 @>{p_{s,t}}>> \EY st 1 \\
@VV{j^*}V @. @| \\
H_{s+t}(\cZ_s(V))  @>{\partial}>> H_{s+t-1}(\cZ_{s-1}(Y)) @<{\delV st}<<
\EY st 1\\
\end{CD}
$$
where $p_{s,t}$ is the projection arising from the decomposition
$$ \EX st 1 =\underset{x\in X_s}{\bigoplus}H_{s+t}(x)=
\EY st 1 \oplus \EV st 1,$$
which comes from the fact $X_s=Y_s\cup V_s$ since $X_s\cap V=V_s$
(cf. \ref{dimension}).
Represent an element of $H_{s+t}(\cZ_s(X))$ by an element of
$H_{s+t}(\bar{W} \cup Z)$, where
$W\in \cZ_s(V)$ with its closure $\bar{W}$ in $X$ and $Z\in \cZ_s(Y)$.
We may enlarge $Z$ to assume $Z\supset \bar{W} \cap Y$, and hence
$\bar{W} \cap Y=\bar{W} \cap Z$.
We write $S=\bar{W} \cup Z$ and $T=\bar{W}  \cap Z$.
We have the localization sequence for the pair $(S,T)$:
$$
H_{s+t}(S) \to H_{s+t}(S-T) \rmapo{\partial_{(S,T)}} H_{s+t-1}(T).
$$
Noting $S-T=(\bar{W}-T)\coprod (Z-T)$, we have the decomposition
$$  H_{s+t}(S-T) =H_{s+t}(\bar{W}-T)\oplus H_{s+t}(Z-T)$$
and then $\partial_{(S,T)}$ is identified with
$\partial_{(\bar{W},T)} \cdot p_W + \partial_{(Z,T)} \cdot p_Z$, where
$$
p_W : H_{s+t}(S-T)\to H_{s+t}(\bar{W}-T)\qaq
p_Z : H_{s+t}(S-T)\to H_{s+t}(Z-T)
$$
are the projections and
$$
\partial_{(\bar{W},T)}:H_{s+t}(\bar{W}-T)\to H_{s+t-1}(T)\qaq
\partial_{(Z,T)}:H_{s+t}(Z-T)\to H_{s+t-1}(T)$$
are the boundary maps for the pairs $(\bar{W},T)$ and $(Z,T)$ respectively.
Thus we get
$$\partial_{(\bar{W},T)} \cdot p_W \cdot \nu+
\partial_{(Z,T)} \cdot p_Z\cdot \nu=0,.$$
where
$\nu: H_{s+t}(S) \to H_{s+t}(S-T)$.
Note $\bar{W}-T=S\cap V$ and $p_W\cdot\nu$ is identified with $j^*$ for
the open immersion $j:S\cap V\hookrightarrow S$. Hence we get
$$\partial_{(\bar{W},T)} \cdot j^* + \partial_{(Z,T)} \cdot p_Z\cdot \nu=0.$$
Now the desired commutativity follows from the following diagram
where all squares are commutative up to sign:
$$
\begin{CD}
\HstXZZ s @<{\piX st}<< \HstXZ s @>{j^*}>> \HstVZ s @. \\
@| @AAA @AAA @. \\
\EX st 1 @<<< H_{s+t}(S) @>{j^*}>> H_{s+t}(\bar{W}-T) @>>> \HstVZ s \\
@VV{p_{s,t}}V @VV{p_Z\cdot\nu}V
@VV{\partial_{(\bar{W},T)}}V @VV{\partial}V \\
\EY st 1 @<<<H_{s+t}(Z-T) @>{\partial_{(Z,T)}}>>
H_{s+t-1}(T) @>>> \HstoYZ {s-1}\\
@| @VVV @VVV @. \\
\HstYZZ s @= \EY st 1 @>{\delY st}>> \HstoYZ {s-1} @.\\
\end{CD}
$$
This completes the proof of \ref{fl-lem1}(1).
\medbreak

\ref{fl-lem1}(2) follows immediately from the definition of $\fF pq r$
and \ref{fl-lem1}(3) from the commutativity of \eqref{gfcomm}.
Finally we show \ref{fl-lem1}(4). It suffices to show the exactness at
$\FF pq r$. In view of the injectivity of
$\EY {p+r}{q-r} 1 @>{i_*}>> \EX {p+r}{q-r} 1$, we have
\begin{align*}
\Ker(\fF pq r)
&= \FF pq r \cap \Ker\big(\HXZ {p+r} @>{\piX pq}>> \EX {p+r}{q-r} 1\big) \\
&= \Ker\big(\frac{\HXZ {p+r-1}}{\KX {p+r-1}{q-r+1} 1} @>{j^*}>>
\frac{\HVZ {p+r}}{\Image(\HVZ p)}\big)
\end{align*}
Consider the commutative diagram
$$
\displaystyle{
\begin{CD}
@. @. \KX {p+r-1}{q-r+1} 1 @>{j^*}>> \KV {p+r-1}{q-r+1} 1 \\
@. @. @VVV @VVV \\
0@>>> \FF pq {r-1} @>>> \HXZ {p+r-1} @>{j^*}>>
\frac{\HVZ {p+r-1}}{\Image(\HVZ p)} \\
@. @VV{\tau}V @VVV @VVV \\
0@>>> \Ker(\fF pq r)  @>>> \frac{\HXZ {p+r-1}}{\KX {p+r-1}{q-r+1} 1} @>{j^*}>>
\frac{\HVZ {p+r}}{\Image(\HVZ p)} \\
\end{CD}
}
$$
where all vertical and horizontal sequences are exact.
Thus, to show the surjectivity of $\tau$ in the diagram, it suffices to prove
the surjectivity of
$\KX {p+r-1}{q-r+1} 1 @>{j^*}>> \KV {p+r-1}{q-r+1} 1$.
By noting
\begin{align*}
\KX {p+r-1}{q-r+1} 1 &= \Image\big(\EX {p+r}{q-r+1} 1 @>{\delX {p+r}{q-r+1}}>>
\HXZ {p+r-1} \big),\\
\KV {p+r-1}{q-r+1} 1 &= \Image\big(\EV {p+r}{q-r+1} 1 @>{\delV {p+r}{q-r+1}}>>
\HVZ {p+r-1} \big),
\end{align*}
it follows from the surjectivity of
$\EX {p+r}{q-r+1} 1 @>{j^*}>> \EV {p+r}{q-r+1} 1$.
This completes the proof of \ref{fl-lem1}.
$\square$

\bigskip

\section{Kato complex of a homology theory}
\bigskip

Let $\cC$ be as in the previous section.
We assume $B=\Spec(k)$ for a field $k$.
Let $\cS\subset \cC$ be the subcategory of smooth projective schemes over $k$.

\begin{defn}\label{ZS} (\cite{GS})
\noindent
\begin{enumerate}
\item[(1)]
Let $\ZS$ (resp. $\CorS$) be the category with the same objects as $\cS$, but with
$$\Hom_{\ZS}(X,Y) =
\underset{i \in I,j\in J}{\bigoplus} \bZ\Hom_{\cS}(X_j,Y_i) $$
$$(\text{resp.} \quad \Hom_{\CorS}(X,Y) =\underset{i \in I,j\in J}{\bigoplus}
\CH^{\dim(Y_i)}(X_j \times Y_i)\quad)$$
for $X,Y\in OB(\cS)$, where
$X_j$ ($j \in J$) and $Y_i$ ($i \in I$) are the connected components of $X$
and $Y$ respectively and
$\bZ\Hom_{\cS}(X_j,Y_i) $ denotes the free abelian group on
$\Hom_{\cS}(X_j,Y_i) $.
It is easy to check that $\ZS$ and $\CorS$ are additive categories and
the coproduct
$X\oplus Y$ of $X,Y\in Ob(\cS)$ is given by $X\coprod Y$. There are natural
functors
\begin{equation}\label{functors}
\cS \to \ZS \to \CorS,
\end{equation}
where the second functor is additive and it maps $f\in \Hom_{\cS}(X,Y) $ to
the class of its graph.
\item[(2)]
For a simplicial object in $\cS$:
$$
X_\bullet\;:\; \quad
\cdots X_2 \;
\begin{matrix}
\rmapo{\delta_0}\\
\lmapo{s_0}\\
\rmapo{\delta_1}\\
\lmapo{s_1}\\
\rmapo{\delta_2}\\
\end{matrix}
\; X_1 \;
\begin{matrix}
\rmapo{\delta_0}\\
\lmapo{s_0}\\
\rmapo{\delta_1}\\
\end{matrix}
\; X_0 ,
$$
we define the complex in $\ZS$:
$$
\bZ X_\bullet\;:\; \quad
\cdots \; \to X_n \rmapo{\partial_n} X_{n-1} \to \cdots\quad
(\partial_n=\sum_{j=0}^n (-1)^j \delta_j).
$$
\item[(3)]
Let $\Lambda$ be a module. To a chain complex in $\CorS$:
$$ X_\bullet\;:\; X_n \rmapo{c_n} X_{n-1} \rmapo{c_{n-1}}
\cdots \to X_1\rmapo{c_1} X_0, $$
we associate a complex of modules called the graph complex of $X_\bullet$:
$$ \graphL {X_\bullet} \;:\;
\Lambda^{\pi_0(X_n)} \rmapo{c_{n*}} \Lambda^{\pi_0(X_{n-1})}
\rmapo{c_{n-1 *}} \cdots \to
\Lambda^{\pi_0(X_1)} \rmapo{c_{1*}} \Lambda^{\pi_0(X_0)}. $$
Here, for $X,Y\in \cS$ connected and for $c\in \CH^{\dim(Y)} (X\times Y)$,
$c_*:\Lambda\to\Lambda$ is the multiplication by
$\sum_{i=1}^N n_i [k(c_i):k(X)]$ where $c=\sum_{i=1}^N n_i c_i$ with
$n_i\in \bZ$ and $c_i\subset X\times Y$,
closed integral subschemes.
For a chain complex $X_\bullet$ in $\ZS$, we let $\graphL{X_\bullet}$
denote the graph complex associated to
the image of $X_\bullet$ in $\CorS$ (cf. \eqref{functors}).
\end{enumerate}
\end{defn}
\medbreak

\begin{defn}\label{defHK}
Fix an integer $e\geq 0$.
\begin{itemize}
\item[(1)]
Let $H$ be a homology theory on $\cC_*$ and let
$$
E^1_{a,b}(X)=\bigoplus_{x\in X_a}H_{a+b}(x) \Rightarrow H_{a+b}(X).
$$
be the niveau spectral sequence associated to $H$. Then
$H$ is \it leveled above $e$ \rm if
\begin{equation}\label{KHcondition}
\EX ab 1 =0 \qforall b<-e \mbox{ and all } X\in Ob(\cC).
\end{equation}
We write $\Lam_H=\Hempty {-e} B$ and call it the coefficient module of $H$.
\item[(2)]
Let $H$ be as in (1). For $X\in Ob(\cC)$ with $d=\dim(X)$, define the Kato
complex of  $X$ by
$$
 \KC X \;:\; \EX d{-e} 1 \to \EX {d-1} {-e} 1 \to \cdots \to \EX 1{-e} 1
\to \EX 0{-e} 1,
$$
where $\EX a {-e} 1$ is placed in degree $a$ and the differentials are
the $d^1$-differentials.
\item[(3)]
We denote by $\KH a X$ the homology group of $\KC X$ in degree $a$ called
the Kato homology of $X$.
By \eqref{KHcondition}, we have the edge homomorphism
\begin{equation}\label{edgehom}
 \edgehom a \;:\; \Hempty {a-e} X \to \KH a X =\EX a {-e} 2.
\end{equation}
\end{itemize}
\end{defn}
\medbreak

\begin{rem}\label{remHK}
If $H$ is leveled above $e$, then the homology theory
$\widetilde{H}=H[-e]$ given by $\widetilde{H}_a(X)=H_{a-e}(X)$ for
$X\in Ob(\cC)$ is leveled above $0$.
Thus we may consider only a homology theory leveled above $0$
without loss of generality.
\end{rem}

In what follows we fix a homology theory $H$ as in \ref{defHK}.

A proper morphism $f:X\to Y$ and an open immersion $j:V\to X$ induce maps
of complexes
$$ f_* : \KC X \to \KC Y,\quad j^*: \KC X \to \KC V,$$
respectively. For a closed immersion $i:Z\hookrightarrow X$ and its complement
$j:V \hookrightarrow X$,
we have the following exact sequence of complexes due to
\eqref{E1exactsequence}:
\begin{equation}\label{KCexactsequence}
0\to \KC Z \rmapo{i_*} \KC X \rmapo{j^*} \KC V \to 0.
\end{equation}
\medbreak

By definition we have
$$ \KC B =\Lam_H[0] \quad (B=\Spec(k))$$
where $\Lam_H [0]$ is the complex with components $\Lam_H$ in degree $0$,
and $0$ in the other degrees.
Thus, if $f: X\to B$ is proper, we get a map of complexes
\begin{equation}\label{KCtrace1}
 f_*\; : \;\KC X \to \Lam_H[0].
\end{equation}
For a chain complex in $\ZS$:
$$ X_\bullet\;:\; X_n \rmapo{f_n} X_{n-1} \rmapo{f_{n-1}} \cdots
\to X_1\rmapo{f_1} X_0 $$
we denote by
$\KC{X_\bullet}$ the total complex of the double complex
$$ \KC{ X_n} \rmapo{f_{n*}} \KC{X_{n-1}} \rmapo{f_{n-1 *}} \cdots
\to \KC{X_1}\rmapo{f_{1*}} \KC{X_0}. $$
The maps \eqref{KCtrace1} for each $n\in \bZ$ induces a natural map of
complexes
called the graph homomorphism:
\begin{equation}\label{graphhomX}
\graphhom{} {X_\bullet} \;:\;  \KC{X_\bullet} \to \graph{X_\bullet}
\end{equation}

\begin{ex}\label{exHK1}
Assume $B=\Spec(K)$ where $K$ is a finite field with the absolute Galois group
$G_K=\Gal(\overline{K}/K)$. Fix a torsion $G_K$-module
$\Lam$, which is viewed as a sheaf on $B_{\et}$.
Taking $\cK=\Lam$ in the example \ref{exH1},
one gets a homology theory $H = H^{\et}(-,\Lam)$ on $\cC$:
$$
H^\et_a(X,\Lam):=H^{-a}(X_{\et}, R\,f^{!}\Lam)
\qfor f:X\rightarrow B \text{ in } \cC.
$$
For $X$ smooth of pure dimension $d$ over $k$, we have
(cf. \cite{BO} and \cite{JS1}, Th.2.14)
\begin{equation}\label{exH1smooth}
H^\et_a(X,\Lam) = H^{2d-a}_{\et}(X,\Lam(d))\,,
\end{equation}
where, for an integer $r>0$, $\Lam(r)$ is defined as follows. If $\Lam$ is
annihilated by an integer $n>1$, define $\Lam(r)=\Lam\otimes \nz(r)$ where $\nz(r)$ is
a bounded complex of sheaves on $X_{\et}$, defined as follows:
Writing $n=m p^t$ with $p=\ch(K)$ and $(p,m)=1$,
\begin{equation}\label{Tatetwist}
\nz(r)=\mu_m^{\otimes r}\oplus \DWtlog r X[-r],
\end{equation}
where $\mu_m$ is the \'etale sheaf of $m$-th roots of unity, and
$\DWtlog r X$ is the logarithmic part of the de Rham-Witt sheaf
$\DWt r X$ \cite{Il}, I 5.7. This definition does not depend on the choice on $n$.
In the general case case we define
$\Lambda(r)=\indlim n \Lambda_n(r)$ where $\Lambda_n=\Ker(\Lambda \rmapo n \Lambda)$.
Here the inductive limit is taken for the transition morphisms
$$
\Lam_n\otimes \nz(r) =\Lam_n\otimes \bZ/n'\bZ(r) \rmapo {i\otimes id} \Lam_{n'}\otimes \bZ/n'\bZ(r)
$$
for $n \mid n'$.

By \eqref{exH1smooth} we get for $X$ general
$$ E^1_{a,b}(X) =\underset{x\in X_a}{\bigoplus} H^{a-b}(x,\Lam(a)).$$
This is a homology theory leveled above $1$:
The condition \eqref{KHcondition} follows from the fact that
$cd(\k(x))=a+1$ for $a\in X_a$ since $cd(K)=1$.
The coefficient module $\Lam_H$ of $H=H^{\et}(-,\Lam)$ is isomorphic to $\Lam$ since
$$H^1(K,\Lam) = \Hom_{cont}(G_K,\Lam) \isom \Lam\; ;\;
\chi \to \chi(Frob_K),$$
where $Frob_K\in G_K$ is the Frobenius substitution.
The arising complex $\KC X$ is written as:
\begin{multline*}\label{KCfinitefield}
\cdots \sumd X a H_{\et}^{a+1}(x,\Lam(a))\to
\sumd X {a-1} H_{\et}^{a}(x,\Lam(a-1))\to \cdots \\
\cdots \to\sumd X 1 H_{\et}^{2}(x,\Lam(1))\to
\sumd X 0 H_{\et}^{1}(x,\Lam).
\end{multline*}
Here the term $\sumd X a$ is placed in degree $i$.
In case $\Lam=\nz$ it is identified up to sign with the complex considered
by Kato in \cite{K} thanks to \cite{JSS}.
\end{ex}

\begin{ex}\label{exHK2}
Assume $B=\Spec(K)$ where $K$ is any field.
Let $G_K$ and $\Lam$ be as in \ref{exHK1} and assume $\Lam$ is finite.
We consider the homology theory $H^D(-,\Lam)$ in the example \ref{exH2}:
$$
\HDL a X:=\Hom\big(H^{a}_c(X,\Lam^\vee),\qz\big) \qfor X\in Ob(\cC).
$$
where $\Lam^\vee=\Hom(\Lam,\qz)$. This homology theory is leveled above $0$:
The condition \eqref{KHcondition} follows from the fact that
$H^q_c(X,\Lam^\vee)=0$ for $q<\dim(X)$ if $X$ is affine scheme over $K$
due to the affine Lefschetz theorem.
The coefficient module $\Lam_H$ of $H=H^{D}(-,\Lam)$ is equal to $\Lam$.
If $K$ is finite, $H^D(-,\Lam)$ shifted by degree $1$ coincides
with $H^{\et}(-,\Lam)$ in \ref{exHK1} due to the Poincar\'e duality for
\'etale cohomology and the Tate duality for Galois cohomology of finite field
(cf. the proof of \ref{HDLE2} below).
\end{ex}

\begin{ex}\label{exHK3}
We will consider the following variants of the homology theories in
\ref{exHK1} and \ref{exHK2}. Fix a prime $l$ and assume given a free
$\zl$-module $T$ of finite rank on which $G_K$ acts continuously.
For each integer $n>1$ put
$$
\Ln=T\otimes\lnz \qaq \Linfty=T\otimes\qzl=\indlim {n} \Ln.
$$
We then consider the homology theories
\begin{equation*}\label{exHK1var}
H^{\et}(-,\Linfty) \qaq H^D(-,\Linfty).
\end{equation*}
For later use, we always assume that we are in either of the following cases:
\begin{itemize}
\item[$(a)$]
$\ell\not=p:=\ch(K)$,
\item[$(b)$]
$K$ is finite and $T=\zp$ on which $G_K$ acts trivially.
\end{itemize}
\end{ex}
\bigskip

For the example \ref{exHK2},
there does not seem to be an evident way to compute the associated Kato
complex in general while we have the following description in case $K$ is
finitely generated over a prime field.
Let $\EX ab r$ be the $E^1$-term associated to the niveau spectral sequence
for the homology theory in \ref{exHK2}.

\begin{prop}\label{HDLE1}
Let the notation be as in \ref{exHK2} and $\EX a b 1$ be $E^1$-term of
the associated spectral sequence.
\noindent
\begin{itemize}
\item[(1)]
Assume $K$ is a finite field. Then we have
$$ \EX a {0} 1 \isom \underset{a\in X_a}{\bigoplus} H^{a+1}(x,\Lam(a))
\qfor X\in \cC.$$
\item[(2)]
Assume $K$ is a global field, namely a number field or a function field in
one variable over a finite field.
Let $\PK$ be the set of the places of $K$ and $K_v$, for $v\in \PK$,
the henselization of $K$ at $v$. Consider the homology theory in \ref{exHK2}.
For a scheme $Z$ over $K$ write
$Z_v=Z\times_{\Spec(K)}\Spec(K_v)$.
Then we have
$$ \EX a {0} 1 \isom \underset{x\in X_a}{\bigoplus} C_x
\qfor X\in \cC,$$
where $C(x)$ ($x\in X_a$) is the cokernel of the diagonal map
$$
H^{a+2}(x,\Lam(a+1)) \to
\underset{v\in \PK}{\bigoplus} H^{a+2}(x_v,\Lam(a+1)).
$$
\item[(3)]
Assume $K$ is the function field of $S$, which is a connected regular proper
flat scheme of relative dimension one over $\Spec(\bZ)$.
We assume for simplicity either that $\Lam$ is annihilated by an odd integer
$n$ or that there is no $\bR$-valued point in $S$.
For $s\in S$ let $A_s$ be the henselization of $\cO_{S,s}$ and
$K_s$ be its field of fractions.
For $\fm\in S_0$ $A_\fm$ is a henselian regular ring of Krull dimension two
and we let $\Pm$ be the set of prime ideals of height one in $A_\fm$.
Let $A_\fp$ for $\fp\in\Pm$ be the henselization of $A_\fm$ at $\fp$ and
$K_\fp$ be its field of fractions.
For a scheme $Z$ over $K$ and for $s\in S$ (resp. $\fp\in \Pm$) write
$Z_s=Z\times_{\Spec(K)}\Spec(K_s)$
(resp. $Z_\fp=Z\times_{\Spec(K)}\Spec(K_\fp)$).
Then we have
$$ \EX a {0} 1 \isom \underset{x\in X_a}{\bigoplus} C_x
\qfor X\in \cC,$$
where $C_x$ ($x\in X_a$) is the cokernel of the diagonal map
$$
\underset{\fm \in S_0}{\bigoplus} H^{a+3}(x_{\fm},\Lam(a+2)) \oplus
\underset{\lambda \in S_1}{\bigoplus} H^{a+3}(x_{\lambda},\Lam(a+2))
\to
\underset{\fm \in S_0}{\bigoplus} \underset{\fp \in \Pm}{\bigoplus}
H^{a+3}(x_{\fp},\Lam(a+2)) .
$$
\end{itemize}
\end{prop}

Note that it is not evident that the image of the above diagonal maps
lies in the direct sum. It is easy but tedious to extend the above result
to the case where $K$ is a general finitely generated field over a prime field
but we do not pursue it in this paper (see for example \cite{KS2}).
\medbreak

Recall that $\Lam$ is a finite $G_K$-module annihilated by an integer $n>1$.
We denote $M^\vee= \Hom(M,\qz)$ for a $\nz$-module $M$.
By definition, \ref{HDLE1} follows from the following:

\begin{prop}\label{HDLE2}
Let $X$ be a connected smooth affine scheme of dimension $d$ over $K$.
\begin{itemize}
\item[(1)]
If $K$ is finite, there is a canonical isomorphism
$$
H^d_c(X,\Lam^\vee)^\vee \simeq H^{d+1}(X,\Lam(d)).
$$
\item[(2)]
If $K$ is a global field, there is a canonical isomorphism
$$
H^{d}_c(X,\Lam^\vee)^\vee\simeq \Coker\big(H^{d+2}(X,\Lam(d+1)) \to
\underset{v\in \PK}{\bigoplus} H^{d+2}(X_v,\Lam(d+1))\big).
$$
\item[(3)]
Let $K$ be as in \ref{HDLE1}(3). There is a canonical isomorphism
$H^{d}_c(X,\Lam^\vee)^\vee \simeq C_X$,
where $C_X$ is the cokernel of the diagonal map
$$
\underset{\fm \in S_0}{\bigoplus} H^{d+3}(X_{\fm},\Lam(d+2)) \oplus
\underset{\lambda \in S_1}{\bigoplus} H^{d+3}(X_{\lambda},\Lam(d+2))
\to
\underset{\fm \in S_0}{\bigoplus} \underset{\fp \in \Pm}{\bigoplus}
H^{d+3}(X_{\fp},\Lam(d+2)).
$$
\end{itemize}
\end{prop}

\begin{proof}
\ref{HDLE2}(1) follows from the Poincar\'e duality for \'etale cohomology
and the Tate duality for Galois cohomology of finite fields.
As for \ref{HDLE2}(2) and (3), we only give the proof of the latter.
The proof of the former is similar and easier.
By SGA$4\frac{1}{2}$, Th. finitude, one can take a dense affine open subscheme
$j:U \hookrightarrow S$ with $n$ invertible on $S$,
and a smooth affine morphism $f:\cX\to U$ such that
$\cX\times_U\eta\simeq X$ and $Rf_*\Lam$ is constructible and commutes with
any base change of $U$. Write $\Sigma_U=S-U$.
For any dense open subscheme $V\subset U$, write
$i_V : \Sigma_V=S-V\hookrightarrow S$.
Let $L=Rf_!\Lam^\vee$ which is an object of $D^b_c(U,\nz)$, the derived
category of bounded complexes of $\nz$-modules whose cohomology sheaves are
constructible. Let $\eta=\Spec(K)$ be the generic point of $S$. Noting
$H^i(\eta,L)\simeq H^i_c(X,\Lam^\vee)$, the localization sequence for \'etale
cohomology provides a long exact sequence
$$
\cdots\to \indlim V H^i_{\Sigma_V}(S,j_! L) \to H^i(S,j_! L) \to
H^i_c(X,\Lam^\vee) \to \indlim V H^{i+1}_{\Sigma_V}(S,j_! L) \to \cdots
$$
Set
$D_U(L)=R\Hom_{D^b(U)}(L,\nz(2))\in D^b_c(U,\nz)$.
By the Poincar\'e duality for the smooth morphism $f$, we have
$Rf^!\nz(2)=\nz(d+2)$ and
$$
D_U(L)\simeq Rf_* R\Hom_{D^b(\cX)}(\Lam^\vee,\nz(d+2)) \simeq
Rf_* \Lam(d+2)[2d]
\; \in D^b_c(U,\nz).$$
By the duality theorem for constructible sheaves on $S$ (\cite{JSS}),
we have canonical isomorphisms
$$
H^i(S,j_!L)^\vee \simeq H^{5-i}(U,D_U(L)) \simeq H^{2d+5-i}(\cX,\Lam(d+2)),
$$
$$
H^i_{\Sigma_V}(S,j_!L)^\vee \simeq H^{5-i}(\Sigma_V, i_V^* Rj_*D_U(L)).
$$
Recalling that $\cX$ is an affine scheme of finite type over $\Spec(\bZ[1/n])$
with $\dim(\cX)=d+2$, $H^{t}(\cX,\Lam(d+2))=0$ for $t>d+4$ by the affine
Lefschetz theorem for arithmetic schemes due to Gabber (cf. \cite{Fu}, \S5).
Therefore we get the exact sequence
$$
H^{d+4}(\cX,\Lam(d+2))\to \projlim V H^{4-d}(\Sigma_V, i_V^* Rj_*D_U(L))
\to H^d_c(X,\Lam^\vee)^\vee\to 0.
$$

\begin{claim}
Writing $\Sigma=\Sigma_U=S-U$, we have a canonical isomorphism
$$
H^{4-d}(\Sigma_V, i_V^* Rj_*D_U(L))\simeq C_{\cX/U}.
$$
Here $C_{\cX/U}$ is the cokernel of the diagonal map
$$
\underset{\fm \in \Sigma_0}{\bigoplus} H^{d+3}(\cX_{\fm},\Lam(d+2)) \oplus
\underset{\lambda \in \Sigma_1}{\bigoplus} H^{d+3}(X_{\lambda},\Lam(d+2))
\to
\underset{\fm \in \Sigma_0}{\bigoplus}
\underset{\fp \in P_{\fm,\Sigma}}{\bigoplus} H^{d+3}(X_{\fp},\Lam(d+2)),
$$
where
$\cX_\fm=\cX\times_S \Spec(A_\fm)$ and $P_{\fm,\Sigma}$ is the subset of
$P_{\fm}$ of those $\fp$ lying over $\Sigma$.
In particular $H^{4-d}(\Sigma_V, i_V^* Rj_*D_U(L))$ is independent of $V$.
\end{claim}
\medbreak

By the claim we get the exact sequence
$$
H^{d+4}(\cX,\Lam(d+2))\to C_{\cX/U} \to H^d_c(X,\Lam^\vee)^\vee\to 0.
$$
Noting $cd(K)=3$, we have $H^{d+4}(X,\Lam(d+2))=0$ by the affine Lefschetz
theorem. Thus \ref{HDLE2}(3) follows by shrinking $U$ to $\eta$.
\end{proof}

\medbreak\noindent
\it Proof of the claim. \rm
Let $\cF=i_V^* Rj_*D_U(L)\in D^b_c(\Sigma_V,\nz)$.
By the localization theory we have the long exact sequence
\begin{multline*}
\underset{\lambda \in (\Sigma_V)_1}{\bigoplus}
H^{t-1}(\Spec(A_\lam),Rj_*D_U(L)) \to
\underset{\fm \in (\Sigma_V)_0}{\bigoplus} H^{t}_\fm(\Sigma_V,\cF) \to
H^t(\Sigma_V,\cF) \\
\to \underset{\lambda \in (\Sigma_V)_1}{\bigoplus}
H^{t}(\Spec(A_\lam),Rj_*D_U(L)).
\end{multline*}
Writing $Y_\lam=\cX\times _U \lam$ for $\lam\in U$, we have
\begin{align*}
H^{t}(\Spec(A_\lam),Rj_*D_U(L))
&\simeq \begin{cases}
      H^{t}(\Spec(K_\lam),D_U(L)) &(\lam\in \Sigma),\\
      H^{t}(\Spec(A_\lam),D_U(L)) &(\lam\in U),
\end{cases} \\
&\simeq \begin{cases}
      H^{2d+t}(X_\lam,\Lam(d+2)) &(\lam\in \Sigma),\\
      H^{2d+t}(Y_\lam,\Lam(d+2)) &(\lam\in U),
\end{cases}
\end{align*}
where we have used the base change property of $R f_*\Lam(d+2)$.
Noting that $\cd(K_\lam)=3$ and $\cd(\kappa(\lam))=2$, the affine Lefschetz
theorem implies
$H^{t}(\Spec(A_\lam),Rj_*D_U(L))=0$ for $t\geq 4-d$ and
a canonical isomorphism
$$
H^{3-d}(\Spec(A_\lam),Rj_*D_U(L)) \simeq
\begin{cases}
      H^{d+3}(X_\lam,\Lam(d+2)) &(\lam\in \Sigma),\\
      0 &(\lam\in U).
\end{cases}
$$
Hence the claim is reduced to establishing a canonical isomorphism, for
$\fm\in (\Sigma_V)_0$:
\begin{equation}\label{HDLEproof2}
H^{4-d}_\fm(\Sigma_V,\cF) \simeq \Coker\big(
H^{d+3}(\cX_{\fm},\Lam(d+2)) \to \underset{\fp \in P_{\fm,\Sigma}}{\bigoplus}
H^{d+3}(X_{\fp},\Lam(d+2))\big).
\end{equation}
For this we use the localization sequence
\begin{multline*}
H^{t-1}(\Spec(A_\fm),Rj_*D_U(L)) \to
\underset{\fp \in P_{\fm,\Sigma_V}}{\bigoplus}
H^{t}(\Spec(A_\fp),Rj_*D_U(L)) \to H^{t}_\fm(\Sigma_V,\cF) \\
\to H^{t}(\Spec(A_\fm),Rj_*D_U(L)).
\end{multline*}
By the same argument as before we get
$$
H^{3-d}(\Spec(A_\fp),Rj_*D_U(L)) \simeq
\begin{cases}
      H^{d+3}(X_\fp,\Lam(d+2)) &(\fp\in P_{\fm,\Sigma}),\\
      0 &(\fp\not\in P_{\fm,\Sigma}).
\end{cases}
$$
In case $\fm\in \Sigma$, we have
\begin{align*}
H^{t}(\Spec(A_\fm),Rj_*D_U(L)) =&
H^{t}(\Spec(A_\fm)\times_S U,Rf_*\Lam(d+2)[2d])\\
 = &H^{2d+t}(\cX\times_S \Spec(A_\fm),\Lam(d+2)).
\end{align*}
In case $\fm\in U$, writing $Y_\fm = \cX \times_U \fm$, we have
$$
H^{t}(\Spec(A_\fm),Rj_*D_U(L)) = H^{t}(\Spec(A_\fm),Rf_*\Lam(d+2)[2d]) =
H^{2d+t}(Y_\fm,\Lam(d+2))
$$
by the base change property of $Rf_*\Lam(d+2)$ and it vanishes
for $t\geq 2-d$ by the affine Lefschetz theorem.
This shows the desired isomorphism \eqref{HDLEproof2} and completes
the proof of the claim.

\bigskip

\section{Statements of the Main Theorems}
\bigskip

Let the notations and assumption be as in the previous section and fix
a homology theory $H$ leveled above $e$ with coefficient module $\Lam_H$.
Recall that $\cS$ is the category of smooth projective schemes over
$B = \Spec(k)$ where $k$ is a field.

\begin{defn}\label{deflogpair}
\noindent
\begin{itemize}
\item[(1)]
A log-pair is a couple $\Phi=(X,Y)$ where $X\in Ob(\cS)$ is connected and
$Y\subset X$ is a simple normal crossing divisor. We call $U=X-Y$ the
complement of $\Phi$ and denote sometime $\Phi=(X,Y;U)$.
A log-pair $\Phi=(X,Y)$ is ample if one of the irreducible components of $Y$
is an ample divisor on $X$.
\item[(2)]
Let  $\Phi=(X,Y;U)$ and $\Phi'=(X',Y';U')$ be two log-pairs.
A map of log pairs $\pi:\Phi' \to\Phi$ is a proper morphism $\pi: X'\to X$
such that $\pi(Y')\subset Y$. It is admissible if $\pi$ induces
an isomorphism $U'=\pi^{-1}(U) \isom U$.
\item[(3)]
Let $\Phi=(X,Y)$ be a log pair and let $Y_1,\dots, Y_N$ be the irreducible
components of $Y$. For an integer $a\geq 1$ write
$$ Y^{(a)} =\underset{1\leq i_i< \cdots< i_a\leq N}{\coprod} Y_{i_1,\dots,i_a}
\quad (Y_{i_1,\dots,i_a} = Y_{i_1}\cap\cdots\cap Y_{i_a}).$$
For $1\leq\nu\leq a$ let
$$ \delta_\nu : Y^{(a)} \to Y^{(a-1)}$$
be induced by the inclusions
$Y_{i_1,\dots,i_a}\hookrightarrow Y_{i_1,\dots,\hat{i_\nu},\dots,i_a}$ and we
define a chain complex in $\ZS$:
$$
\Phi_\bullet=(X,Y)_\bullet \;:\; Y^{(d)} \rmapo{\partial} Y^{(d-1)}
\rmapo{\partial} \cdots \rmapo{\partial} Y^{(1)}
\rmapo{\iota} X,
\quad (d=\dim(X))
$$
where
$$ \partial=\sum_{i=1}^a (-1)^\nu \delta_\nu \;:\; Y^{(a)} \to Y^{(a-1)}$$
and $\iota$ is induced by the inclusion $Y\hookrightarrow X$.
We denote by $Cor(\Phi_\bullet)$ the associated complex in $\CorS$.
A map of log pairs $\pi:\Phi' \to\Phi$ induces a map
$\pi_*:\Phi'_\bullet \to \Phi_\bullet$ of complexes in $\ZS$.
\end{itemize}
\end{defn}
\bigskip

For a log-pair $\Phi=(X,Y;U)$ it is easy to check that the natural
map of complexes
\begin{equation}\label{KCFU}
\KC {\Phi_\bullet} \isom \KC U
\end{equation}
is a quasi-isomorphism. Combined with the map of complexes
$\KC {\Phi_\bullet} \rightarrow \graph {X_\bullet}$ (cf. \eqref{graphhomX})
we get natural maps
\begin{equation}\label{graphhomPhi}
\graphhom {a}  {\Phi_\bullet} \;:\; \KH a U \to \graphH a {\Phi_\bullet},
\end{equation}
where the right hand side is the homology in degree in $a$ of $\graph {\Phi_\bullet}$.
Let
\begin{equation}\label{graphedgehomPhi}
\graphedge a \Phi \;:\; \Hempty {a-e} U \to \graphH a {\Phi_\bullet}
\end{equation}
be the composite of the above map with the edge homomorphism \eqref{edgehom}.

\begin{defn}\label{def.clean1}
A log pair $\Phi=(X,Y;U)$ is $H$-clean in degree $q$ for an integer $q$
if $q\leq \dim(X)$ and $\graphedge a \Phi$ is injective for $a=q$ and
surjective for $a=q+1$.
\end{defn}

We now consider the following condition (called the Lefschetz condition)
for our homology theory $H$:
\begin{enumerate}
\item[{\bf (L)} ]  : Every ample log pair is $H$-clean in degree $q$ for
all $q \leq \dim(X)$.
\end{enumerate}

\begin{lem}\label{lemL1}
A homology theory $H$ leveled above $e$
satisfies the Lefschetz condition, if the following conditions holds:
\begin{itemize}
\item[$(H1)$]
For $f:X\to B=\Spec(k)$, smooth projective of dimension $\leq 1$
with $X$ connected (but not necessarily geometrically irreducible over $B$),
$f_*\;:\; \Hempty {-e} X \to \Hempty {-e} B=\Lam_H$
is an isomorphism if $\dim(X)=0$ and injective if $\dim(X)=1$.
\item[$(H2)$]
For $X$, projective smooth of dimension $>1$ over $B$, and  $Y\subset X$,
an irreducible smooth ample divisor, and $U=X-Y$, one has
$$ \Hempty {a-e} {U} =0\qfor a\leq d=\dim(X).$$
\item[$(H3)$]
For a projective smooth curve $X$ over $B$ and for a dense affine open
subset $U\subset X$,
$$ \Hempty {a-e} {U} =0\qfor a\leq 0$$
and $\Hempty {1-e} U @>{\partial}>> \Hempty {-e} Y$ is injective, where
$Y=X-U$ with the reduced subscheme structure.
\end{itemize}
\end{lem}
\medbreak

\begin{lem}\label{lemL2}
Consider the homology theories $H^{\et}(-,\Linfty)$ and $H^D(-,\Linfty)$
in \ref{exHK3}. In the case $(a)$ of \ref{exHK3}, assume the following:
\begin{itemize}
\item[$(i)$]
$K$ is finitely generated over a prime field and
$T$ is mixed of weights $\leq 0$ (\cite{D}),
\item[$(ii)$]
$T^{G_K} = T^{G_L}$ for any finite separable extension $L/K$ with
$G_L=\Gal(\overline{K}/L) \subset G_K$.
\end{itemize}
Then they satisfy the conditions in
\ref{lemL1} and hence the Lefschetz condition.
\end{lem}
\medbreak
The proofs of \ref{lemL1} and \ref{lemL2} will be given in the last part of
this section.
\bigskip

We restate {$\bf (RES)_q$} in the introduction. Let $q\geq 0$ be an integer.

\begin{enumerate}
\item[{$\bf (RES)_q$}] :
For any log pair $(X,Y;U)$ and for any irreducible closed
subscheme $W\subset X$ of dimension $\leq q$ such that $W\cap U$ is regular,
there exists an admissible map of log-pairs $\pi: (X',Y') \to (X,Y)$ such that
the proper transform of $W$ in $X'$ is regular and intersects transversally
with $Y'$.
\end{enumerate}
\medbreak

$\bf (RES)_{q}$ holds if $\ch(F)=0$ by Hironaka's theorem.
It is shown in general for $q=2$ in \cite{CJS}.

\begin{thm}\label{mainthmI}
Let $H$ be a homology theory leveled above $e$ which satisfies {$\bf(L)$}.
Let $q\geq 1$ be an integer and assume {$\bf(RES)_{q-2}$}.
Then, for any log-pair $\Phi$, the map induced by \eqref{graphhomPhi}:
$$
\graphhom {a}  {\Phi_\bullet} \;:\; \KH  a U \to \graphH a
{\Phi_\bullet}.
$$
is an isomorphism for all $a\leq q$. In particular, if $X\in Ob(\cS)$
$$ \KH a X =0\qfor 0< a\leq q.$$
\end{thm}
\medbreak

The proof of \ref{mainthmI} will be completed in the next section.
\bigskip

We also consider a variant of the main theorem \ref{mainthmI},
where we replace ${\bf {(RES)_q}}$ by a condition {$\bf (RS)_d$} introduced below.
Let $d\geq 1$ be an integer and let $\cC_d\subset \cC$ be the full
subcategory of the schemes of dimension$\leq d$.

\begin{enumerate}
\item[{$\bf (RS)_d$}] :
For any $X\in Ob(\cC_d)$ integral and proper over $k$,
there exists a proper birational morphism
$\pi: X'\to X$ such that $X'$ is smooth over $k$.
For any $U\in Ob(\cC_d)$ smooth over $k$, there is an open immersion
$U\hookrightarrow X$ such that $X$ is projective smooth over $k$ with $X-U$,
a simple normal crossing divisor on $X$.
\end{enumerate}
\medbreak

For an additive category $\cA$ we denote by $\Hot(\cA)$ the category of
complexes in $\cA$ up to homotopy. It is triangulated by defining the triangles
to be the diagram isomorphic in $\Hot(\cA)$ to diagrams of the form:
$$ A_\bullet \rmapo{f} B_\bullet \to Cone(f) \to A_\bullet[1],$$
where $f$ is any morphism of complexes in $\cA$.
\medbreak

Now assume {$\bf(RS)_d$}.
Fix $X\in Ob(\cC_d)$ and take a compactification $j:X\to \ol X$, namely
$j$ is an open immersion and $\ol X\in \cC$ which is proper over $k$.
Let $i:Y=\ol X -X \to \ol X$ be the closed immersion for the complement.
By \cite{GS} 1.4, one can find a diagram
\begin{equation}\label{hyperenv}
\begin{CD}
Y_\bullet @>{i_\bullet}>> \ol X_\bullet \\
@V{\pi_Y}VV @VV{\pi_X}V \\
Y @>{i}>> \ol X \\
\end{CD}
\end{equation}
where $Y_\bullet$ and $\ol X_\bullet$ are simplicial objects in $\cS$ and
$\pi_X$ and $\pi_Y$ are hyperenvelopes.
To this diagram one associates
$$
\wtWC X:=[Y_\bullet \rmapo{i_\bullet} \ol X_\bullet]:=
Cone\big(\bZ Y_\bullet \rmapo{{i_\bullet}_*} \bZ \ol X_\bullet \big)
\in \Hot(\ZS).
$$
The weight complex of $X$:
$$\WC X =Cor([Y_\bullet \rmapo{i_\bullet} \ol X_\bullet])\in \Hot(\CorS)$$
is defined as the image of $\tWC X$ under $Cor: \ZS \to \CorS$.
By the definition of hyperenvelopes we have
a natural quasi-isomorphism of complexes of abelian groups
\begin{equation}\label{WCKC}
\KC { [Y_\bullet \rmapo{i_\bullet} \ol X_\bullet] } \isom \KC X.
\end{equation}
By \cite{GS}, 1.4, we have the following facts:

\begin{thm}\label{WCGS}
Assume {$\bf(RS)_d$} and that all schemes are in $\cC_d$.
\begin{itemize}
\item[(1)]
Up to canonical isomorphism, $\WC X$ depends only on $X$ and not on a choice of the diagram
\eqref{hyperenv}.
\item[(2)]
A proper morphism $f:X\to Y$ and an open immersion $j:V\to X$ induce canonical maps
in $\Hot(\CorS)$
$$ f_* : \WC X \to \WC Y,\quad j^*: \WC X \to \WC V.$$
For a closed immersion $i:Z\hookrightarrow X$ and its complement
$j:V \hookrightarrow X$, there is a natural distinguished triangle in $\Hot(\CorS)$
$$ \WC Z \rmapo{i_*} \WC X \rmapo{j^*} \WC V \to \WC Z [1].$$
\end{itemize}
\end{thm}

By extending the results in \cite{GS}, the following is shown in
\cite{J1} 5.13, 5.15 and 5.16.

\begin{thm}\label{WCJ}
Assume {$\bf(RS)_d$} and that all schemes are in $\cC_d$.
\begin{itemize}
\item[(1)]
There is a canonical homology theory $X \longmapsto \graphH \ast X$ on $\cC_d$
such that
\begin{equation}\label{graphhomologytheory}
\graphH a X = \graphH a {\WC X} \quad (a\in \bZ),
\end{equation}
and the localization sequences are induced by the exact triangles
in \ref{WCGS} (2).
\item[(2)]
For any log-pair $\Phi = (X,Y;U)$ one has $\graphH a U = \graphH a \Phi_\bullet$.
\item[(3)]
There is a canonical morphism of homology theories on $\cC_d$
$$ \graphhom \ast - \;:\; \KH \ast - \to \graphH \ast - $$
such that for any log-
pair $\Phi = (X,Y;U)$ the map
\begin{equation}\label{KCGH}
 \graphhom a U \;:\; \KH a U \to \graphH a U
\end{equation}
coincides with the map defined in \eqref{graphhomPhi}.
\end{itemize}
\end{thm}


By definition, in the situation of \eqref{hyperenv} one has
$$\graphH a X = \graphH a {[Y_\bullet \rmapo{i_\bullet} \ol X_\bullet]},$$
and the maps $\graphhom a X$ are induced by the natural map of complexes
\begin{equation}\label{eq.hommap}
\KC { [Y_\bullet \rmapo{i_\bullet} \ol X_\bullet] }
\longrightarrow \graph {[Y_\bullet \rmapo{i_\bullet} \ol X_\bullet]}
\end{equation}
together with the quasi-isomorphism \eqref{WCKC}.
For a closed subscheme
$i:Z\hookrightarrow X$ and its open complement $j:V\hookrightarrow X$ we have
the commutative diagram
\begin{small}
\begin{equation}\label{CDloc}
\CD
\KH {a+1} V @>>> \KH {a} Z @>>> \KH {a} X @>>> \KH {a} V  @>>>\\
@VV{\graphhom {a+1} V}V @VV{\graphhom {a} Z}V
@VV{\graphhom {a} X }V  @VV{\graphhom a V}V  @.\\
\graphH {a+1} V @>>> \graphH a Z  @>>> \graphH a X @>>> \graphH a V  @>>> \\
\endCD
\end{equation}
Since the cone is not a well-defined functor in the homotopy category, this diagram
does not directly follow from Theorem \ref{WCGS}, but by following the construction
in \cite{GS} more closely.
\end{small}

\begin{thm}\label{mainthmII}
Let $H$ be a homology theory leveled above $e$ which satisfies {$\bf(L)$} and
admits correspondences after restriction to $\cS_d$ (see Theorem \ref{WCJ} (3)).
Assume {$\bf(RS)_d$}. For any $X\in Ob(\cC_d)$ we have
$$
\graphhom a X \;:\; \KH a X \isom \graphH a X \qforall  a.
$$
In particular, if $X\in Ob(\cS)$ of dimension $\leq d$,
$$ \KH a X =0
\qforall  a\geq 1.
$$
\end{thm}

The proof of \ref{mainthmII} will be completed in the next section. We will now prove
Lemmas \ref{lemL1} and \ref{lemL2}. Here and later we will use the following result.

\begin{lem}\label{lemPhi2}
Let $(X,Y;U)$ be a log-pair.
Let $\iota: Z\hookrightarrow X$ be a smooth prime divisor such that
$(X,Z\cup Y)$ is a log-pair.
Note that it implies that $(Z,Y\cap Z)$ is a log-pair and $\iota$ induces a
map of log-pairs
$(Z,Z\cap Y) \to (X,Y)$. Then there is a natural isomorphism of complexes in $\ZS$:
$$ Cone\big( (Z,Z\cap Y)_\bullet \rmapo{\iota_*} (X,Y)_\bullet)
\isom (X, Z\cup Y)_\bullet.$$
\end{lem}

\begin{proof} There are direct sum decompositions in $\ZS$
$$ (Y\cup Z)^{(i)} = Y^{(i)} \oplus \;Y^{(i-1)}\cap Z ,$$
where the right-hand side is the $i$-th component of the cone, and it is
easily checked that the differentials coincide.
\end{proof}

\bigskip\noindent
\it Proof of Lemma \ref{lemL1}. \rm
\medbreak
By shift of degree we may assume $H$ is leveled above $0$.
Let $\Phi=(X,Y;U)$ be an ample log pair with $d=\dim(X)$.
Let $Y_1,\dots,Y_N$ be the irreducible components of $Y$ and assume
$Y_1$ is an ample divisor on $X$. We want to show
\begin{equation}\label{L1sub1}
\graphedge a \Phi\;:\;
\Hempty a U \simeq  \graphH a {\Phi_\bullet}\qfor a\leq d=\dim(X).
\end{equation}
First assume $d=1$ so that $X$ is a projective smooth curve over $B=\Spec(k)$
and $Y$ is smooth of dimension $0$. Then it is easy to see
$$
\graphH a {\Phi_{\bullet}}=0\qfor a\not=1,\quad
\graphH 1 {\Phi_{\bullet}}\simeq
\Ker\big(\Lam_H^{\pi_0(Y)} \to \Lam_H^{\pi_0(X)}\big).
$$
Considering the exact sequence
$ \Hempty 1 U \rmapo{\partial} \Hempty 0 Y \to \Hempty 0 X ,$
\eqref{L1sub1} in this case follows from $(H1)$ and $(H3)$.
Next assume that $d>1$ and $N=1$.
In this case it is obvious from the definition that
$\graphH a {\Phi_\bullet}=0$ for all $a$.
Hence \eqref{L1sub1} follows from $(H2)$.
Finally we prove \eqref{L1sub1} in general by induction on $N$.
We may assume $d>1$ and $N>1$.
Write
$Z= Y_1\cup \cdots\cup Y_{N-1}$ and
consider the log pairs $\Psi=(X,Z;V)$ and $\Psi'=(Y_N,Y_N\cap Z;W)$.
There is a natural map of log pairs
$\Psi' \to \Psi$ induced by $Y_N\hookrightarrow X$ and by Lemma \ref{lemPhi2}
we have $\Phi_\bullet \simeq Cone(\Psi'_\bullet \to \Psi_\bullet)$,
which induces the lower exact sequence in the following commutative diagram
\begin{footnotesize}
$$
\begin{CD}
\Hempty a W @>>> \Hempty a V @>>> \Hempty a U @>>> \Hempty {a-1} W @>>> \Hempty {a-1} V \\
@V{\graphedge a {\Psi'}}VV @V{\graphedge a \Psi}V{\simeq}V
@V{\graphedge a \Phi}VV @V{\graphedge {a-1} {\Psi'}}V{\simeq}V
@V{\graphedge {a-1} {\Psi}}V{\simeq}V \\
\graphHempty a {\Psi'_\bullet} @>>>\graphHempty a {\Psi_\bullet} @>>>
\graphHempty a {\Phi_\bullet} @>>> \graphHempty {a-1} {\Psi'_\bullet} @>>>
\graphHempty {a-1} {\Psi_\bullet} \\
\end{CD}
$$
\end{footnotesize}
For $a\leq d=\dim(X)$, the isomorphisms in the diagram follow from the
induction hypothesis. The leftmost map $\graphedge a {\Psi'}$
is an isomorphism for $a\leq d-1$ by the induction hypothesis and surjective
for $a=d$ since $\graphHempty {d} {\Psi'_\bullet}=0$ by reason of
dimension. A diagram chase proves \eqref{L1sub1} and the proof of
\ref{lemL1} is complete.

\bigskip\noindent
\it Proof of Lemma \ref{lemL2}. \rm
\medbreak

In case $(b)$ of \ref{exHK3} we only have to consider $H^{\et}(-,\qzp)$
by \ref{HDLE2}(1). Then $(H1)$ is obvious and the other conditions
are shown by the same argument as the proof of \cite{JS1}, Theorem 3.5.
The details are left to the readers.
Assume we are in the case $(a)$ of \ref{exHK3}.
By \ref{HDLE2}(1)  it suffices to consider only $H^D(-,\Linfty)$.
$(H1)$ follows easily from the second assumption in \ref{lemL2}.
In order to show $(H2)$, let $f:X\to B=\Spec(K)$ be
geometrically irreducible smooth projective of dimension $d > 1$
and let $Z\subset X$ be a smooth ample divisor with $U=X-Z$.
Let $\Kb$ be a separable closure of $K$ and $G_K =\Gal(\Kb/K)$.
For a scheme $W$ over $K$, write $\otkb W=W\times_K \Kb$.
Since $U$ is affine by the assumption, the affine Lefschetz theorem implies
$H^i_c(\otkb U,\Ln^\vee) =0$ for $i<d$.
By the Hochschild-Serre spectral sequence:
$$
E_2^{a,b}= H^a(G_K,H^b_c(\otkb U,\Ln^\vee)) \rightarrow H^{a+b}_c(U,\Ln^\vee),
$$
it implies $\HDLn a U=0$ for $a\leq d-1$ and
$H^d_c(U,\Ln^\vee)\simeq H^d_c(\otkb U,\Ln^\vee)^{G_K}$.
By the Poincar\'e duality
$$
\HDLn d U \simeq \Hom(H^d_c(\otkb U,\Ln^\vee)^{G_K},\qz)
\simeq H^d(\otkb U,\Ln(d))_{G_K},
$$
where $M_{G_K}$ is the module of coinvariants of $G_K$ for a
$G_K$-module $M$. Thus we have to show the vanishing of the last group.
By the localization theory we have the exact sequence
\begin{equation}\label{LemL2sub2}
H^d(\otkb X,\Linfty(d)) \to H^d(\otkb U,\Linfty(d)) \to
H^{d-1}(\otkb Z,\Linfty(d-1)).
\end{equation}
By the affine Lefschetz theorem
$H^d(\otkb U,\Linfty(d))=H^d(\otkb U,T(d))\otimes\qzp$ is divisible.
By Deligne's fundamental result \cite{D} the first assumption in \ref{lemL2}
implies that $H^d(\otkb X,T(d))$ is of weights $\leq -d$ and
$H^{d-1}(\otkb Z,T(d-1))$ of weights $\leq -(d-1)$. Hence
$H^d(\otkb U,\Linfty(d))_{G_K}=0$ noting $d>1$. This proves $(H2)$.

Finally we show $(H3)$. Let $X,U,Y$ be as in $(H3)$.
By the localization theory we have the exact sequence
\begin{equation}\label{LemL2sub1}
0\to H^1(\otkb X,\Linfty(1)) \to H^1(\otkb U,\Linfty(1)) \rmapo{\partial}
H^0(\otkb Y,\Linfty) \to \Linfty \to 0.
\end{equation}
where we have used the trace isomorphism
$H^2(\otkb X,\Linfty(1))\simeq\Linfty$.
We have the commutative diagram
$$
\begin{CD}
\HDLinfty 1 U @>{\partial}>> \HDLinfty 0 Y \\
@V{\alpha}V{\simeq}V @V{\beta}V{\simeq}V \\
H^1(\otkb U,\Linfty(1))_{G_K} @>{\gamma}>> \Linfty^{\pi_0(Y)}\\
\end{CD}
$$
where $\alpha$ is an isomorphism shown by the same argument as before
and $\beta$ is the sum of the isomorphisms
$\HDLinfty 0 y\simeq\HDLinfty 0 B\simeq \Linfty$ for all points $y\in Y$.
The map $\gamma$ is the composite of $\partial$ in \eqref{LemL2sub1} and
the natural isomorphism
$$ H^0(\otkb Y,\Linfty)_{G_K} \simeq \Linfty,$$
which follows from the identification
\begin{equation}\label{LemL2sub3}
H^0(\otkb Y,\Linfty) =\underset{x\in Y}{\bigoplus} \Ind_{K(x)/K}\Linfty,
\end{equation}
where $K(x)$ is the residue field of $x\in Y$ and
$\Ind_{K(x)/K} \Linfty$ denotes the $G_K$-module induced from the
$G_{K(x)}$-module $\Linfty$ ($G_{K(x)}=\Gal(\Kb/K(x))\subset G_K$).
Thus it remains to show that $\partial$ in \eqref{LemL2sub1} induces
an injection after taking coinvariants for $G_K$.
Since $H^1(\otkb X,\Linfty(1))$ is divisible and $H^1(\otkb X,T(1))$ is
of weights $\leq -1$ by \cite{D}, we have $H^1(\otkb X,\Linfty(1))_{G_K}=0$.
Hence \eqref{LemL2sub1} induces an isomorphism
$H^1(\otkb U,\Linfty(1))_{G_K} \simeq \Image(\partial)_{G_K}$
and the exact sequence
$$
0\to \Image(\partial) \to H^0(\otkb Y,\Linfty) \to \Linfty \to 0.
$$
We thus need to show that the last exact sequence remains exact after
taking the coinvariants for $G_K$. In view of \eqref{LemL2sub3}, there exists
a finite Galois extension $L/K$ such that the above sequence splits as
a sequence of $G_L$-modules. Hence
$$
0\to \Image(\partial)_{G_L} \to H^0(\otkb Y,\Linfty)_{G_L}  \to
(\Linfty)_{G_L}  \to 0
$$
is exact and it remains so after taking the coinvariant of $G_{L/K}:=G_K/G_L$
due to the divisibility of $\Linfty$.
This proves the desired assertion and the proof of \ref{lemL2} is complete.

\bigskip

\section{Proof of the Main Theorems}
\bigskip

Let the assumption be as in the previous section.
In this section we prove the main theorems \ref{mainthmI} and \ref{mainthmII}.
We start with \ref{mainthmII}.
Its proof is much simpler and conveys the basic idea more clearly.

\begin{defn}\label{def.clean2} (Compare \ref{def.clean1})
$X\in Ob(\cC_d)$ is $H$-clean in degree $q$ if $q\leq \dim(X)$ and
the composite of \eqref{KCGH} and \eqref{edgehom}:
$$\graphedge {a} X\;:\; \Hempty {a-e} X \to \graphH a X$$
is injective for $a=q$ and surjective for $a=q+1$.
\end{defn}

By the definition of $\graphhom a X$ it suffices to show \ref{mainthmII}
in case $X\in Ob(\cS)$.
Then, by the commutative diagram \eqref{CDloc}, it suffices to show \ref{mainthmII}
for $X-Y$ where $Y\subset X$ is a smooth hypersurface section.
Since $X-Y$ is $H$-clean in degree $a$ for all $a\leq\dim(X)$ if
$H$ satisfies {$\bf(L)$}, the assertion follows from the following theorem.

\begin{thm}\label{mainthmIIv}
Let $H$ be a homology theory leveled above $e$ which satisfies
{$\bf(L)$} and admits correspondences after restriction to the category of
smooth projective varieties in $\cC_d$ (see Theorem \ref{WCJ} (3)), and let
$$
E^1_{a,b}(X)=\bigoplus_{x\in X_a}H_{a+b}(x)~~\Rightarrow ~~H_{a+b}(X)
$$
be the niveau spectral sequence associated to $H$. Assume {$\bf(RS)_d$}.
If $X\in Ob(\cC_d)$ is $H$-clean in degree $q-1$ for an integer $q\geq 0$,
we have
$$
\ZBX ab \infty {b+e} =0 \quad\text{if $a+b=q-1-e$ and $b\geq 1-e$}.
$$
\end{thm}
\medbreak

In fact, Theorem \ref{mainthmII} for $X$ is deduced as follows: By the factorization
\begin{equation}\label{eq.gammaepsilon}
\graphedge {a} X\;:\; \Hempty {a-e} X \rmapo {\edgehom a} E^2_{a,-e}(X) = \KH a X \rmapo {\graphhom a X} \graphH a X,
\end{equation}
the $H$-cleanness of $X$ in all degrees $a \leq \dim(X)$ and the fact that $H$ is leveled above $e$
imply that $E^\infty_{a,b}(X)=0$ for $b \geq 1-e$. Moreover Theorem \ref{mainthmIIv} implies
that the differentials
\begin{equation}\label{eq.ZB}
d^r_{a,-e}: E^r_{a,-e}(X) \rightarrow
\ZBX {a-r} {-e+r-1} \infty r \subseteq E^r_{a-r,-e+r-1}(X)
\end{equation}
are zero for all $r \geq 3$. Thus $\edgehom a$ above is an isomorphism, and so is $\graphhom a X$, as claimed in
\ref{mainthmII}.

\medbreak\noindent\textit{Proof of Theorem \ref{mainthmIIv}:}\;
By shift of degree we may assume $e=0$. Fix an integer $q\geq 0$.
In what follows we write for an integer $l\geq 1$
\begin{equation}\label{Theta}
\Th l X = \ZBX {q-l-1}{l} \infty l \qfor X\in Ob(\cC).
\end{equation}
We prove $\Th l X =0$ for all $l\geq 1$ by induction on $\dim(X)$ and
by (descending) induction on $l$.
In case that $l$ sufficiently large the assertion is obvious.
The assumption that $X$ is $H$-clean in degree $q-1$ implies $\dim(X)\geq q-1$
and, by using \eqref{eq.gammaepsilon} as before, that the edge homomorphism \eqref{edgehom}:
$$ \edgehom {q-1} : \Hempty {q-1} X \to \KH {q-1} X=
\EX {q-1}{0}  2$$
is injective and hence that $\EX a b  {\infty}=0$ if $a+b=q-1$ and
$b\geq 1$.
In case $\dim(X)=q-1$ it implies the desired assertion by noting that
$\EX ab 1=0$ if $b<0$ (cf. \ref{defHK} (1)) and that $\EX q {0} 1=0$ by reasons
of dimension. Assume $\dim(X)\geq q$ and fix $t\geq 1$.
By induction it suffices to show $\Th t X=0$ under
the following assumption.
\begin{enumerate}
\item[$(*)$] :
For $X'\in Ob(\cC)$, $H$-clean in degree $q-1$, $\Th l {X'}=0  $ if
$\dim(X')<\dim(X)$ or $l\geq t+1$.
\end{enumerate}

Choose $\alpha\in \Th t X$.
By definition there exists a closed subscheme $W\subset X$ with
$\dim(W)=q-t-1\leq q-2<\dim(X)$ such that the restriction of $\alpha$ to
$\Th t {U-W}$ vanishes. Thus it suffices to show the following:

\begin{claim}\label{mainclaimII}
Let $X$ be as above.
Let $W\subset X$ be any closed subscheme with $\dim(W)<\dim(X)$.
Then there exists a closed subscheme $W\subset Z\subset X$ with
$\dim(Z)<\dim(X)$ satisfying the following:
\begin{itemize}
\item[(1)]
$V:=X-Z$ is $H$-clean in degree $\leq q$.
\item[(2)]
The induce map
$j^*: \Th t X \to \Th t V$
is injective, where $j:V\to X$ is the open immersion.
\end{itemize}
\end{claim}
\bigskip

\begin{proof}
First we show that \ref{mainclaimII} (1) implies (2).
Consider the commutative diagram:
\begin{small}
$$
\CD
\Hempty {i+1} V @>>> \Hempty {i} Z @>>> \Hempty {i} X @>>>
\Hempty {i} V @>>> \\
@VV{\graphhom {i+1} V}V @VV{\graphhom {i} Z}V
@VV{\graphhom {i} X}V @VV{\graphhom {i} V}V  \\
\graphH {i+1} V @>>> \graphH i Z  @>>> \graphH i X @>>>
\graphH i V @>>> \\
\endCD
$$
\end{small}
By the assumption on $X$, $\graphhom {i} X$ is injective for $i=q-1$
and surjective for $i=q$. By \ref{mainclaimII}(1) $\graphhom {i} V$ is
injective for $i\leq q$ and surjective for $i=q+1$.
The diagram chase now shows that $\graphhom i Z$ is injective for $i=q-1$
and surjective for $i=q$ so that $Z$ is $H$-clean in degree in $q-1$.
By the induction hypothesis $(*)$ we have
$\Th l {V}=0 $ if $l\geq t+1$ and $\Th l {Z}=0$ for
$\forall l\geq 1$. By the fundamental lemma \ref{FL} this implies
$\ref{mainclaimII}(2)$.

Now we prove (1). We may clearly assume that $X$ is reduced.
By {$\bf (RS)_d$} there is a dense open subscheme $U\subset X-W$
such that there is a compactification $U\hookrightarrow \ol {X}$ such that
$\ol {X}\in \cS$ and that $Y:=\ol {X}-U$ is a simple normal crossing divisor
on $\ol {X}$. By Bertini's theorem (here we use \cite{P} if the base field
is finite), we can find a smooth hypersurface section
$Z' \subset \ol {X}$ such that $(\ol {X},Y\cup Z')$ is an ample log-pair.
Put $V=U\backslash Z'=\ol {X}-(Y\cup Z')$ and $Z=X-V$.
By the construction it is obvious that $\dim(Z)<\dim(X)$ and $W\subset Z$.
Since $(\ol {X},Y\cup Z';V)$ is an ample log-pair, the assumption {$\bf (L)$}
implies that $V$ is $H$-clean in degree$\leq q$ by noting $\dim(V)=\dim(X)\geq q$.
This completes the proof.
$\square$
\end{proof}
\bigskip

Next we prove \ref{mainthmI}.
The basic idea is the same as in the proof of \ref{mainthmII} but the application is
more technical.
By the same argument as before, the proof is reduced to showing the following:

\begin{thm}\label{mainthmIv}
Let $H$ be a homology theory leveled above $e$ which satisfies {$\bf(L)$}.
Let $q\geq 1$ be an integer and assume {$\bf(RES)_{q-2}$}.
If a log-pair $\Phi=(X,Y;U)$ is $H$-clean in degree $q-1$ (Definition \ref{def.clean1}),
we have
$$
\ZBU ab \infty {b+e} =0 \quad\text{if $a+b=q-1-e$ and $b\geq 1-e$}.
$$
\end{thm}
\begin{proof}
By shift of degree we may assume $e=0$. Let the notations be as \eqref{Theta}.
As before we prove $\Th l U =0$ for all $l\geq 1$ by induction on
$\dim(U)$ and by (descending) induction on $l$.
For $l$ sufficiently large or for the case $\dim(U)\leq q-1$ the assertion
can be shown in the same way as before.
Assume $\dim(U)\geq q$ and fix $t\geq 1$.
By induction it suffices to show $\Th t U=0$ under
the following assumption.
\begin{enumerate}
\item[$(**)$] :
For a log-pair $\Phi'=(X',Y';U')$, $H$-clean in degree $q-1$,
$\Th l {U'}=0  $ if $\dim(U')<\dim(U)$ or $l\geq t+1$.
\end{enumerate}

Choose $\alpha\in \Th t U$.
By definition there exists a closed subscheme $W\subset U$ with
$\dim(W)=q-t-1\leq q-2<\dim(X)$ such that the restriction of $\alpha$ to
$\Th t {U-W}$ vanishes.
We note that there is a stratification
$W \supset W_1\supset\cdots\supset W_M$ with $W_i$
closed in $U$ such that $W_j-W_{j+1}$ is irreducible regular.
Hence it suffices to show the following.
\end{proof}

\begin{claim}\label{mainclaim}
Let $\Phi=(X,Y;U)$ be as above. Assume $\dim(U)\geq q$.
Let $W\subset X$ be an irreducible closed subscheme of dimension $\leq q-2$
such that $W_U:=W\cap U$ is regular. Assume {$\bf (RES)_{q-2}$}.
Then there exists a log pair $\Psi=(X',Y';V)$ satisfying the following:
\begin{itemize}
\item[(1)]
$\Psi$ is $H$-clean in degrees $\leq q$.
\item[(2)]
There is an open immersion $j:V\hookrightarrow U$ such that $W_U\subset U-V$.
\item[(3)]
The induced map
$j^*: \Th t U \to \Th t V$
is injective.
\end{itemize}
\end{claim}
\medbreak

We need some preliminaries for the proof of the claim.

\begin{lem}\label{lemPhi1}
Let $\Phi=(X,Y)$ and $\Phi'=(X',Y')$ be log-pairs.
Let $\pi:\Phi'\to\Phi$ be an admissible map of log-pairs.
Then the induced map
$ \pi_*:  Cor(\Phi_\bullet) \to Cor(\Phi'_\bullet)$
is an isomorphism in $\Hot(\CorS)$.
\end{lem}

\begin{lem}\label{lemPhi3}
Let $(X,Y;U)$ be a log-pair.
Let $\iota: W\hookrightarrow X$ be a closed irreducible smooth subscheme and
assume that
$(W, W\cap Y)$ is a log-pair.
Let $\pi_X : \tX\to X$ be the blowup of $X$ along $W$, and let
$\tY, E \subset\tX$ be the proper transform of $Y$
and the exceptional divisor, respectively. Let
$$ i_E: E\to \tX, \quad \pi_E:E\to W, \quad i_W:W\to X$$
be the natural morphisms. Then $(\tX,\tY)$ and $(E,E\cap \tY)$
are log-pairs, and there is a natural isomorphism in $\Hot(\ZS)$:
$$ Cone\big( (E,E\cap \tY)_\bullet \rmapo{ ({i_E}_*, -{\pi_E}_*) }
(\tX,\tY)_\bullet \oplus (W,W\cap Y)_\bullet\big)
\rmapou{\pi_{X*} +i_{W*}}{\simeq}  (X,Y)_\bullet. $$
\end{lem}
\medbreak

Both Lemmas follow from \cite{GS}, theorem 1.
In fact, with the notation of loc. cit. we have a complex of complexes
$$
R_{q,\ast,V}(X,Y): \quad \ldots  \rightarrow  R_{q,\ast}(V\times Y^{(2)}) \rightarrow  R_{q,\ast}(V\times Y^{(1)})  \rightarrow  R_{q,\ast}(V\times X) $$
for every log pair $(X,Y)$, every $q \geq 0$ and every smooth projective variety $V$. For
Lemma \ref{lemPhi1} it suffices to show that the canonical morphism
$R_{q,\ast,V}(X',Y') \rightarrow R_{q,\ast,V}(X,Y)$ induces a quasi-isomorphism
of the associated total complexes for all $q$ and $V$. Then \cite{GS} Theorem 1 (see its Corollary 1) implies the claim.
But it is easy to see that one has an exact sequence for every log pair $(X,Y;U = X - Y)$
\begin{equation}\label{eq.4.6}
\ldots  \rightarrow  R_{q,\ast}(V\times Y^{(2)}) \rightarrow  R_{q,\ast}(V\times Y^{(1)})  \rightarrow  R_{q,\ast}(V\times X) \rightarrow R_{q,\ast}(V\times U) \rightarrow 0 \,,
\end{equation}
i.e., a morphism $totR_{q,\ast,V}(X,Y) \rightarrow R_{q,\ast}(V\times U)$ of complexes
which is a quasi-isomorphism.
Since $X' - Y' = X - Y$ in Lemma \ref{lemPhi1}, the claim follows.

As for Lemma \ref{lemPhi3}, one has a commutative diagram of complexes in $\ZS$:
\begin{equation}\label{eq.4.7.A}
\begin{CD}
(E,E\cap\wt Y)_\bullet @>{i_{E,\ast}}>> (\wt X,\wt Y)_\bullet \\
@VV{\pi_{E,\ast}}V  @VV{\pi_{X,\ast}}V  \\
(W,W \cap Y)_\bullet @>{i_{W,\ast}}>> (X,Y)_\bullet \;,
\end{CD}
\end{equation}
and we have to show that the associated total complex has a contracting homotopy.
By \cite{GS} Thm. 1 it suffices to show that for each $q \geq 0$ and each $V \in \cS$
the induced commutative diagram
\begin{equation}\label{eq.4.7.B}
\begin{CD}
R_{q,\ast,V}(E,E\cap\wt Y) @>{i_{E,\ast}}>> R_{q,\ast,V}(\wt X,\wt Y) \\
@VV{\pi_{E,\ast}}V  @VV{\pi_{X,\ast}}V  \\
R_{q,\ast,V}(W,W \cap Y) @>{i_{W,\ast}}>> R_{q,\ast,V}(X,Y)
\end{CD}
\end{equation}
has the property that the cone of the upper line is
quasi-isomorphic to the cone of the lower line.
\smallskip
But by \eqref{eq.4.6} the upper cone is quasi-isomorphic to the cone of
$$
R_{q,\ast}(V\times (E - (E \cap \wt Y))) \rightarrow R_{q,\ast}(V\times (\wt X - \wt Y))\,,
$$
so by the obvious exact sequence
$$
0 \rightarrow R_{q,\ast}(V\times (E - E \cap \wt Y)) \rightarrow R_{q,\ast}(V\times (\wt X - \wt Y))
\rightarrow R_{q,\ast}(V \times (\wt X - (E \cup \wt Y))) \rightarrow 0
$$
the upper cone is quasi-isomorphic to $R_{q,\ast}(V \times (\wt X - (E \cup \wt Y)))$.
Similarly, the lower cone is quasi-isomorphic to $R_{q,\ast}(V\times (X - (W \cup Y)))$,
and one checks that one has a commutative diagram
$$
\begin{CD}
Cone(i_{E,\ast}) @>>> R_{q,\ast}(V \times (\wt X - (E \cup \wt Y))) \\
@VVV @VVV \\
Cone(i_{W,\ast}) @>>> R_{q,\ast}(V\times (X - (W \cup Y)))\,,
\end{CD}
$$
in which the vertical morphisms are induced by morphisms $\pi_{E,\ast}$ and $\pi_{X,\ast}$
and the functoriality of the Cone on the left, and the projection
$\pi': \wt X - (E \cup \wt Y) \rmapo \sim X - (W \cup Y)$ induced by $\pi_X$ on the right.
Now the claim follows, because $\pi'$ is an isomorphism.

\medskip
Now we start the proof of \ref{mainclaim}.
By {$\bf (RES)_{q-2}$} and \ref{lemPhi1} we may assume that $W$ is regular of
dimension $\leq q-2$
intersecting transversally with $Y$. Consider the following diagrams:
$$
\begin{matrix}
E &\hookrightarrow& \tX &\hookleftarrow& \tY \\
\downarrow&\square&\downarrow\rlap{$\pi$}&&\downarrow\\
W &\hookrightarrow& X &\hookleftarrow& Y \\
\end{matrix}
\quad\text{and}\quad
\begin{matrix}
\tX &\hookleftarrow& \tU &\hookleftarrow& E_U \\
\downarrow&\square&\downarrow&\square&\downarrow\\
X &\hookleftarrow& U &\hookleftarrow& W_U \\
\end{matrix}
$$
where $\tX$ is the blowup of $X$ along $W$, $E$ is the exceptional divisor,
$\tY$ is the proper transform of $Y$ in $X$, $\tU= \tX - \tY$, $W_U=W\cap U$,
$E_U=E\cap \tU$.
Note that $\tY\cup E$ is a simple normal crossing divisor on $\tX$.
By Bertini's theorem (as extended to finite fields by Poonen \cite{P}) we can find
hypersurface sections $H_1,\cdots,H_N\subset \tX$ with $N=\dim(X)-q+1$
(recall that we have assumed $\dim(X)\geq q$)
such that $\tY\cup E\cup H_1\cup\cdots\cup H_N$ is a simple normal crossing
divisor on $\tX$. Then the morphism $E_\nu = E\cap H_1\cap\cdots\cap H_\nu \to W$
is surjective for $\nu = 0, \ldots, N$. In fact, it follows by induction that the
fibers are of dimension $\geq \dim(X) - \nu -1 - \dim(W) = N - \nu + q - 2 - \dim(W)
\geq N - \nu \geq 0$: This holds for $E = E_0$, and if shown for $E_\nu$ with $\nu < N$
it follows for $E_{\nu+1}$, because the fibers of $E_\nu$ are proper of dimension
$> 0$ and contain the fibers of $E_{\nu+1} = E_\nu\cap H_{\nu+1}$ as non-empty
divisors, because $E_\nu - (E_\nu\cap H_{\nu+1})$ is affine, and so are its fibers.
We get the diagram:
$$
\begin{matrix}
\llap{$\tU =:$ }\tU_0 &\hookleftarrow& \tU_1 &\hookleftarrow& \tU_2
&\hookleftarrow& \cdots &\hookleftarrow& \tU_N \\
\downarrow&&\downarrow&&\downarrow&&\hbox{}&&\downarrow \\
\llap{$\tX =:$ }\tZ_0 &\hookleftarrow& \tZ_1 &\hookleftarrow& \tZ_2
&\hookleftarrow& \cdots &\hookleftarrow& \tZ_N \\
\downarrow&&\downarrow&&\downarrow&&\hbox{}&&\downarrow \\
\llap{$X =:$ }Z_0 &\hookleftarrow& Z_1 &\hookleftarrow& Z_2 &\hookleftarrow&
\cdots &\hookleftarrow& Z_N  \\
\uparrow&&\uparrow&&\uparrow&&\hbox{}&&\uparrow \\
\llap{$U =:$ }U_0 &\hookleftarrow& U_1 &\hookleftarrow& U_2 &\hookleftarrow&
\cdots &\hookleftarrow& U_N  \\
\end{matrix}
$$
where $\tZ_\nu=H_1\cap\cdots\cap H_\nu$, $Z_\nu$ is its image in $X$,
$\tU_\nu=\tZ_\nu\cap \tU$, $U_\nu=Z_\nu\cap U$ for $1\leq \nu \leq N$.
Let
$$\tY_\nu=\tZ_\nu \cap \tY, \; \; E_\nu=\tZ_\nu\cap E \;\; (0\leq \nu\leq N),\;\;
V_\nu=\tZ_\nu-(\tZ_{\nu+1}\cup \tY_\nu\cup E_\nu) \;\; (0\leq \nu\leq N-1).$$
We note that $V_\nu=U_\nu - U_{\nu+1}$ for $0\leq \nu \leq N$ and that
$\dim(Z_\nu)=q+N-\nu-1$ and that $Z_\nu$ is regular off $W$ but may be
singular along it.
We have the following log pairs:
$$
\Psi_\nu=(\tZ_\nu, \tZ_{\nu+1}\cup \tY_\nu\cup E_\nu; V_\nu), \;
(\tZ_\nu, \tY_\nu; \tU_\nu), \;
(E_\nu, E_\nu \cap \tY; E_\nu\cap \tU), \;
(W, W \cap Y; W\cap U).
$$
We claim that $\Psi_0=(\tX,\tZ_1\cup \tY\cup E)$ satisfies the desired
properties of \ref{mainclaim}.
Indeed \ref{mainclaim}(1) follows from the assumption {\bf (L)} since
$H_1=\tZ_1$ is an
ample divisor on $\tX=\tZ_0$.
\ref{mainclaim}(2) follows from the fact that $W\subset Z_N$ and
$V_0=U\backslash Z_1$.
It remains to show \ref{mainclaim}(3). We set
$$
{\Phi_\nu}_\bullet =
Cone\big((E_\nu,E_\nu\cap \tY)_\bullet
\rmapo{ ({i_{E_\nu}}_*, {\pi_{E_\nu}}_*) } (\tZ_\nu,\tY_\nu)_\bullet
\oplus (W,W\cap Y)_\bullet \big),
$$
where
$i_{E_\nu} : E_\nu\to \tZ_\nu$ and $\pi_{E_\nu} : E_\nu \to W$ are the natural
morphisms. There is a natural morphism
\begin{equation}\label{mainclaim1}
\KCLH {{\Phi_\nu}_\bullet} \isom \KCLH {U_\nu}
\end{equation}
which is a quasi-isomorphism for $0 \leq n \leq N$. In fact, we have a commutative diagram
$$
\begin{CD}
\KCLH {(E_\nu,E_\nu\cap \tY)_\bullet} @>{{i_{E_\nu}}_\ast}>> \KCLH {(\tZ_\nu,\tY_\nu)_\bullet} \\
@V{{\pi_{E_\nu}}_\ast}VV @VV{{\pi_{Z_\nu}}_\ast}V \\
\KCLH {(W,W\cap Y)_\bullet} @>{{i_{W_\nu}}_\ast}>> \KCLH {Z_\nu,Z_\nu\cap Y} \;,
\end{CD}
$$
and, by \eqref{KCFU}, the upper row is quasi-isomorphic to
$$
\KCLH {E_\nu - (E_\nu\cap\tY_\nu)} \rightarrow \KCLH {\tZ_\nu - \tY_\nu}\;,
$$
while the lower row is quasi-isomorphic to
$$
\KCLH {W - (W \cap Y)} \rightarrow \KCLH {U_\nu}\;.
$$
Now the claim follows, because by \eqref{KCexactsequence} the associated total complexes are quasi-isomorphic to
$ \KCLH {\tZ_\nu - (E_\nu\cup \tY_\nu)}$ and $\KCLH {Z_\nu - (W\cup Y_\nu)}$,
respectively, and $\pi$ induces an isomorphism $\tZ_\nu - (E_\nu\cup \tY_\nu) \cong
Z_\nu - (W\cup Y_\nu)$.

By \ref{lemPhi3} we have the natural isomorphism
\begin{equation}\label{mainclaim2}
{\Phi_0}_\bullet  \isom (X,Y)_\bullet \quad \text{in } \Hot(\ZS).
\end{equation}
Moreover we claim that there are natural isomorphisms
\begin{equation}\label{mainclaim3}
Cone\big( {\Phi_{\nu+1}}_\bullet \rmapo{\iota_*\oplus id_W}
{\Phi_\nu}_\bullet \big) \isom {\Psi_\nu}_\bullet
\end{equation}
in the category $C(\ZS)$ of complexes in $\ZS$ where
$\iota :\tZ_{\nu+1}\to \tZ_\nu$ is the natural morphism.
Indeed, for a morphism of complexes $f: A \rightarrow B$ call the natural sequence
of complexes $A \rmapo f B \rightarrow Cone(f)$ a cone sequence. Then we have
the following commutative diagram in $C(\ZS)$:
$$
\begin{CD}
(E_{\nu+1}, E_{\nu+1}\cap \tY)_\bullet @>>> (\tZ_{\nu+1},\tY_{\nu+1})_\bullet
\oplus (W,W\cap Y)_\bullet @>>> {\Phi_{\nu+1}}_\bullet \\
@VVV  @VVV @VVV \\
(E_\nu, E_\nu\cap \tY)_\bullet @>>>  (\tZ_\nu,\tY_\nu)_\bullet \oplus
(W,W\cap Y)_\bullet @>>>
{\Phi_{\nu} }_\bullet \\
@VVV  @VVV @. \\
(E_\nu,  (E_\nu\cap \tY)\cup E_{\nu+1} )_\bullet @>>>  (\tZ_\nu, \tY_\nu \cup
\tZ_{\nu+1})_\bullet @>>>
  (\tZ_\nu, \tY_\nu \cup \tZ_{\nu+1} \cup E_\nu )_\bullet \\
\end{CD}
$$
where the two left vertical sequences and the bottom horizontal sequence are cone sequences
by \ref{lemPhi2} and by noting that
$(E_\nu\cap \tY)\cup E_{\nu+1} = E_\nu\cap (\tY_\nu \cup \tZ_{\nu+1}) $.
Now \eqref{mainclaim3} follows from the following elementary lemma.

\begin{lem}\label{elm.lemma}
Consider a diagram of cone sequences (in any additive category $\cA$)
$$
\begin{CD}
A @>{f}>> B @>>> Cone(f) \\
@VV{a}V @VV{b}V @V{c}VV \\
A' @>{f'}>> B' @>>> Cone(f') \\
@VVV @VVV @. \\
Cone(a) @>{d}>> Cone(b) @.
\end{CD}
$$
in which the morphisms $c$ and $d$ come from the functoriality of the cone.
Then there is a canonical isomorphism $Cone(c) \rmapo \sim Cone(d)$ in the
category of complexes in $\cA$.
\end{lem}

To wit: In degree $n$ it is $(B')^n \oplus (A')^{n+1} \oplus B^{n+1} \oplus A^{n+2}
\rightarrow (B')^n \oplus B^{n+1} \oplus (A')^{n+1} \oplus A^{n+2}$ given by
``$(b',a',b,a) \mapsto (b',b,a',-a)$''.

\medbreak
By
\eqref{mainclaim3},
we get the following commutative diagram with exact rows
(the coefficients $\Lam_H$ are omitted):
$$
\begin{matrix}
\cdots & \KH {i+1} {V_\nu} &\to& \KH {i} {U_{\nu+1}} &\to& \KH {i} {U_\nu}  &\to&
\KH {i} {V_\nu} &\cdots\\
&\downarrow&&\downarrow&&\downarrow&&\downarrow\\
\cdots & \KH {i+1} {{\Psi_\nu}_\bullet} &\to& \KH {i} {{\Phi_{\nu+1}}_\bullet} &\to& \KH {i} {{\Phi_\nu}_\bullet} &\to&
\KH {i} {{\Psi_\nu}_\bullet} &\cdots \\
&\downarrow&&\downarrow&&\downarrow&&\downarrow\\
\cdots & \graphHempty {i+1} {{\Psi_\nu}_\bullet} &\to&
\graphHempty i {{\Phi_{\nu+1}}_\bullet} &\to&
\graphHempty i {{\Phi_\nu}_\bullet} &\to&
\graphHempty i {{\Psi_\nu}_\bullet} &\cdots \;.\\
\end{matrix}
$$
Here the upper vertical maps come from the quasi-isomorphism \eqref{KCFU} and \eqref{mainclaim1},
and the upper long exact sequence comes from the exact sequence of complexes
$$
0 \rightarrow \KC {U_{\nu+1}} \rightarrow \KC {U_\nu} \rightarrow \KC {V_\nu} \rightarrow 0
$$
due to and \eqref{KCexactsequence} and the fact that $V_\nu = U_\nu - U_{\nu+1}$.

By composing with the edge homomorphisms \eqref{edgehom} we get
the commutative diagram with exact rows
\begin{footnotesize}
$$
\CD
H_{i+1}(U_\nu) @>>>H_{i+1}(V_\nu) @>>> H_{i}(U_{\nu+1}) @>>> H_{i}(U_\nu)
@>>> H_{i}(V_\nu) \\
@VV{\graphhom {i+1} {\Phi_\nu} }V
@VV{\graphhom {i+1} {\Psi_\nu} }V @VV{\graphhom {i} {\Phi_{\nu+1}} }V
@VV{\graphhom {i} {\Phi_\nu} }V @VV{\graphhom {i} {\Psi_\nu} }V \\
\graphHempty {i+1} {{\Phi_\nu}_\bullet } @>>>
\graphHempty {i+1} {{\Psi_\nu}_\bullet } @>>>
\graphHempty i {{\Phi_{\nu+1}}_\bullet } @>>>
\graphHempty i {{\Phi_\nu}_\bullet } @>>>
\graphHempty i {{\Psi_\nu}_\bullet }  \\
\endCD
$$
\end{footnotesize}
In view of \eqref{mainclaim2} the assumption that $(X,Y)$ is $H$-clean in
degree $q-1$ implies that
$\graphhom i {\Phi_0}$ is injective for $i=q-1$ and surjective for $i=q$.
Since $\Psi_\nu$ is ample and $\dim(V_\nu)\geq q$ ($\nu\leq N-1$),
$\graphhom i {\Psi_\nu}$ is an isomorphism for $i\leq q$ and surjective
for $i=q+1$ by the assumption {\bf (L)}. The diagram chase now shows that
for all $q$ with $0\leq \nu\leq N$, $\graphhom i {\Phi_\nu}$ is injective
for $i=q-1$ and surjective for $i=q$. Then the following facts hold:
\begin{enumerate}
\item[$(*1)$]
$\Th l {V_\nu}=0$ for all $l\geq 1$ and for all
$1\leq \nu\leq N-1$.
\item[$(*2)$]
$\Th l {V_0} =0$ for all $l\geq t+1$,
\item[$(*3)$]
$\Th l {U_N}=0$ for all $l\geq 1$.
\end{enumerate}
$(*1)$ and $(*2)$ follow from the induction hypothesis $(**)$ by noting
that $\dim(V_\nu)<\dim(U)$ if $\nu\geq 1$.
$(*3)$ holds since $\dim(U_N)=q-1$ and $\gamma^{q-1}_{\Phi_N}$ is injective
(cf. the argument in the first step of the induction).
Recall $V_\nu=U_\nu - U_{\nu+1}$ for $0\leq \nu \leq N$.
By the fundamental lemma \ref{FL}, $(*1)$ and $(*3)$ imply that
$\Th l {U_\nu}=0$ for $\forall l\geq 1$ and for $1\leq\forall \nu\leq N$.
By \ref{FL} this assertion for $\nu=1$ together with $(*2)$ implies the
injectivity of
$\Th t {U} \to \Th t {V_0}$, which proves \ref{mainclaim}(3).

\bigskip

\section{Results with finite coefficients}
\bigskip

The main results in \S3 show, under the assumption of resolution of
singularities, the vanishing of the Kato homology of a projective
smooth variety for a certain homology theory with infinite coefficient
module $\Linfty$ (see \ref{exHK3}).
In this section we improve it to the case of finite coefficient modules $\Ln$.

Fix a rational prime $\ell$. Assume given an inductive system of
homology theories:
$$H=\{\HLLn,\iota_{m,n}\}_{n\geq 1},$$
where $H(-,\Ln)$ are homology theories leveled above $e$ on
$\cC$, a category of schemes over the base $B=\Spec(k)$.
It gives rise to a homology theory
$$ H(-,\Linfty)\;:\; X \to H_{a}(X,\Linfty):=\indlim {n\geq 1} H(X,\Ln)
\qfor X\in Ob(\cC)$$
with
$\iota_n: H(-,\Ln) \to H(-,\Linfty)$, a functor of homology theories.
We assume that it induces an exact sequence for each $n\geq 1$:
\begin{equation}\label{LHes}
0 \to H_0(B,\Ln) @>{\iota_n}>> H_0(B,\Linfty) @>{\nt}>> H_0(B,\Linfty)\to 0.
\end{equation}
We further assume given, for each integer $n\geq 1$,
a map of homology theories of degree $-1$
\begin{equation}\label{boundaryL}
\partial_n\;:\; H(-,\Linfty) \to H(-,\Ln)
\end{equation}
such that for any $X\in Ob(\cC)$ and for any integers $m>n$,
we have the following commutative diagram of exact sequences
\begin{equation}\label{ESfinite}
\begin{CD}
0 @>>> \Hinfty {q+1} X/\nt @>{\partial_n}>> \Hn q X
@>{\iota_n}>> \Hinfty q X [\nt] @>>> 0 \\
@. @| @VV{\iota_{m,n}}V @|\\
0@>>> \Hinfty {q+1} X/\mt @>{\partial_m}>> \Hn q X
@>{\iota_m}>> \Hinfty q X [\mt] @>>> 0. \\
\end{CD}
\end{equation}
\medbreak

We let $\KHn a X$ and $\KHinfty a X$ denote the Kato homology
associated to $H(-,\Ln)$ and $H(-,\Linfty)$ respectively. By definition
$\KHinfty a X =\indlim {n\geq 1} \KHn a X.$

\begin{rem}
The homology theories $\{\H^{\et}(-,\Ln)\}_{n\geq 1}$ and
$\{\H^{D}(-,\Ln)\}_{n\geq 1}$ in \ref{exHK3} satisfy
the above assumption.
\end{rem}
\medbreak

We now consider the following condition for $H(-,\Linfty)$:

\begin{enumerate}
\item[{$\Dql$}] :
For any $X\in Ob(\cC)$ which is connected regular of dimension $q$
with $\eta\in X_q$, the generic point,
$\Hinfty {q-e+1} \eta$ is divisible by $\ell$.
\end{enumerate}

\begin{rem}
\par\noindent
\begin{itemize}
\item[(1)]
For the homology theory in \ref{exHK1} the condition {$\Dql$}
is implied by the Bloch-Kato conjecture. We will explain this later
in this section.
\item[(2)]
In view of \eqref{ESfinite} {$\Dql$} is equivalent to the injectivity
of $ \HLn {q-e} {\eta} \to \Hinfty {q-e} {\eta}$, which implies
the injectivity of $\KHn q X \to \KHinfty q X$ for $X$ connected regular
of dimension $q$ since by definition $\KHn q X$ is a subgroup of
$\Hn {q-e} \eta$.
\end{itemize}
\end{rem}
\medbreak

Let
$$
E^1_{a,b}(X,\Ln)=\bigoplus_{x\in X_a}H_{a+b}(x,\Ln)~~\Rightarrow ~~
H_{a+b}(X,\Ln)
$$
be the niveau spectral sequence associated to $H(-,\Ln)$.

\begin{thm}\label{mainthmIfinite}
Let $q,d\geq 1$ be integers.
Assume that $H=H(-,\Linfty)$ satisfies {$\bf(L)$} and {$\Dql$}.
Assume either {$\bf(RES)_{q-2}$} or {$\bf(RS)_{d}$}.
\begin{itemize}
\item[(1)]
Let $\Phi=(X,Y;U)$ be a log-pair with $\dim(X)\leq d$ and assume that it is
$H$-clean in degree $q-1$and $q$. Then we have for any integer $n\geq 0$
$$
\ZBUn ab \infty {b+e} =0 \quad\text{if $a+b=q-1-e$ and $b\geq 1-e$}.
$$
\item[(2)]
For any $X\in Ob(\cS)$ of dimension $\leq d$,
$ \KHn q X =0 $ for any integer $n\geq 0$.
\end{itemize}
\end{thm}
\begin{proof}
By shift of degree we may assume $e=0$.
First we prove (1). Recall that the cleanness of $\Phi$ in degree $q$ implies
$q\leq \dim(U)$. Once (1) is shown in case $\dim(U)=q$, then the case
$\dim(U)\geq q+1$ is shown by the same argument as in the proof of
\ref{mainthmIv} and \ref{mainthmIIv}. We thus treat the case $\dim(U)=q$.
It suffices to show that
$$ \graphedgen {q-1} \Phi \;:\; \Hn {q-1} U \to \graphHn {q-1} {\Phi_\bullet}$$
is injective and that the edge homomorphism
$$ \epsilon^q_U  \;:\; \Hn q U \to \KHn q U $$
is surjective. To show the first assertion, we consider the commutative daigram
$$
\begin{CD}
0 @>>> \Hinfty {q} U /\nt @>{\partial_n}>> \Hn {q-1} U @>>>
\Hinfty {q-1} U [\nt] @>>> 0 \\
@. @V{\simeq}V{\graphedge {q} \Phi}V @VV{\graphedge {q-1} \Phi}V
@V{\hookrightarrow}V{\graphedge {q-1} \Phi}V \\
0@>>> \graphHinfty {q} {\Phi_\bullet}/\nt @>>> \graphHn {q-1} {\Phi_\bullet} @>>> \graphHinfty {q-1} {\Phi_\bullet}[\nt] @>>> 0\\
\end{CD}
$$
Here $\graphHn \bullet {-}$ and $\graphHinfty \bullet {-}$ are
the graph homologies associated to $\HLLn$ and $\HLLinfty$
(cf. \eqref{graphhomX}), respectively, and the lower exact sequence comes from \eqref{LHes}.
The commutativity of the left square follows from the assumption that
$\partial_n$ (cf. \eqref{boundaryL}) is a map of homology theories.
The left vertical arrow is an isomorphism and the right vertical one is
injective by the assumption that $\Phi$ is $H$-clean in degree $q-1$ and $q$.
The desired assertion follows from this.
In order to show the second assertion, we consider the commutative diagram
$$
\begin{CD}
\Hn q U @>\alpha>> \Hinfty q U [\nt]  \\
@VV{\epsilon^q_U}V @V{\simeq}V{\epsilon^q_U} V \\
 \KHn q U @>\beta>> \KHinfty q U [\nt] \\
\end{CD}
$$
The right vertical arrow is an isomorphism due to \ref{mainthmI} and
\ref{mainthmII} in view of the assumption that $\Phi$ is clean in degree $q$.
The map $\alpha$ is surjective by \eqref{ESfinite}. Noting $\dim(U)=q$,
{$\Dql$} implies that $\beta$ is injective. This shows the desired
surjectivity.

We now deduce \ref{mainthmIfinite}(2) from (1).
We may assume that $X$ is connected of dimension $\geq q$. Assume $\dim(X)=q$.
{$\Dql$} implies $\KHn q X\hookrightarrow \KHinfty q X$ and thus the
assertion follows from \ref{mainthmI} and \ref{mainthmII}.
Assume $\dim(X)> q$ and proceed by induction on $\dim(X)$.
Let $Y\subset X$ be a smooth hyperplane section and
consider the log-pair $\Phi=(X,Y;U)$ with $U=X-Y$. By induction $\KHn q Y=0$
and the exact sequence \eqref{KCexactsequence} implies
$\KHn q X\hookrightarrow \KHn q U$. Thus it suffices to show
$\KHn q U=0$. Since $\Phi$ is clean in degree $q-1$ and $q$ by {$\bf (L)$},
\ref{mainthmIfinite}(1) implies the edge homomorphism
$\Hn q U \to \KHn q U$ is surjective so that it suffices to show
$\Hn q U=0$. Consider the commutative diagram
$$
\begin{CD}
0 @>>> \Hinfty {q+1} U /\nt @>{\partial_n}>> \Hn q U @>>> \Hinfty q U [\nt]
 @>>> 0 \\
@. @V{\simeq}V{\graphedge {q+1}\Phi}V @VV{\graphedge {q}\Phi}V
@V{\simeq}V{\graphedge {q}\Phi}V \\
0@>>> \graphHinfty {q+1} {\Phi_\bullet}/\nt @>>> \graphHn q {\Phi_\bullet} @>>> \graphHinfty q {\Phi_\bullet}[\nt] @>>> 0\\
\end{CD}
$$
The left and right vertical arrows are isomorphisms by {$\bf (L)$}
and the assumption $\dim(U)=\dim(X) > q$.
By definition $\graphHn a {\Phi_\bullet}=0$ for all $a\geq 1$.
This shows the desired assertion and completes the proof of
\ref{mainthmIfinite}. $\square$
\end{proof}
\bigskip

In the rest of this section we consider the homology theory in
Example \ref{exHK1}:
We take the base $B=\Spec(F)$ for a finite field $F$.
For an integer $n\geq 0$ define
$$
H^{\et}_a(X,\nz):=H^{-a}(X_{\et}, R\,f^{!}\nz)
\qfor f:X\rightarrow B \text{ in } \cC.
$$
This homology theory is leveled above $e=1$ and the Kato complex $\KC X$ for
$X\in Ob(\cC)$ is the complex \eqref{KCfinitefieldintro} in Introduction.
We have
\begin{equation}\label{Hetsmooth}
\H^{\et}_a(X,\nz)= H^{2q-a}(X,\nz(q))
\text{ for $X\in Ob(\cC)$ smooth over $B$ of dimension $q$}.
\end{equation}
\medbreak

Now apply Theorem \ref{mainthmIfinite} to the inductive system
$\{H^{\et}(-,\Ln)\}_{n\geq 1}$ with $\Ln=\lnz$.
By \eqref{Hetsmooth}, if $X$ is regular and connected with $\eta\in X_q$,
the generic point, we have
$$
\Hetinfty {q-e+1} \eta=\Hetinfty {q} \eta= H_{\et}^q(\eta,\qzl(q)):=
\indlim n H_{\et}^q(\eta,\lnz(q)).$$
One easily sees that the surjectivity of the symbol map for a field $L$:
$$h^q_{L,\ell} : K^M_q(L) \to H^q(L,\lz(q))$$
implies $H^q(L,\qzl(q))$ is $l$-divisible.
Hence condition {$\BKlz {q}$} in the introduction implies {$\Dql$} in this case.
Therefore \ref{mainthmIfinite} implies the following:

\begin{thm}\label{mainthmKC1}
Let $X$ be projective smooth of dimension $d$ over a finite field $F$.
Let $t\geq 1$ be an integer.
Assume either $t\leq 4$ or {$\RS {d}$}, or {$\RES {t-2}$}.
Assume further $\BKlz t$. Then we have for any integer $n\geq 0$
$$
\KHetlnz a X  \simeq \left.\left\{\gathered
 \lnz \\
 0 \\
\endgathered\right.\quad
\aligned
&\text{$a=0$}\\
&\text{$0< a \leq t$}
\endaligned\right.
$$
\end{thm}

\begin{cor}\label{mainthmKC2}
Let $X$ be a separated scheme of finite type of dimension $d$ over
a finite field $F$. Let $q,n\geq 1$ be integers.
Assume {$\bf(RS)_{d}$} and let
$$ \graphhom a X \;:\; \KHetnz a X \to \graphHnz a X$$
be the map \eqref{KCGH} defined for the \'etale homology theory
$H^{\et}(-,\nz)$. Assume $\BKlz t$ for all primes $l|n$.
Then $\graphhom a X$ is an isomorphism for $\forall a\leq t$.
\end{cor}
\begin{proof}
By the assumed resolution of singularities and the commutative diagrams \eqref{CDloc},
\ref{mainthmKC2} is reduced to the case where $X$ is smooth projective, which follows
from \ref{mainthmKC1}.
$\square$
\end{proof}
\bigskip

Recall that \ref{mainthmIfinite} shows not only the vanishing of Kato homology
for $X$ smooth projective but also that of $\ZBU ab \infty {b+e}$ for an ample
log-pair $\Phi=(X,Y;U)$.
In order to see the consequences of this more clearly, we look at $E^1$-terms in lower
degrees associated to the homology theory $H^{\et}(-,\Lam)$ with $\Lam=\nz$:
\begin{equation*}\label{sslowdegree}
\begin{matrix}
 & \deg\;0 &&\deg\;1 &&\deg\;2 &&\cdots \\
\\
\EUnz {\bullet} 1 1\; :\quad &
0 &\gets& \sumd U 1 H_{\et}^{0}(x,\Lam(1)) &\gets&
\sumd U 2 H_{\et}^{1}(x,\Lam(2)) &\gets& \cdots \\
\\
\EUnz {\bullet} 0 1\; :\quad &
\sumd U 0 H_{\et}^{0}(x,\Lam) &\gets& \sumd U 1 H_{\et}^{1}(x,\Lam(1)) &\gets&
\sumd U 2 H_{\et}^{2}(x,\Lam(2)) &\gets& \cdots \\
\\
\EUnz {\bullet} {-1} 1\; :\quad &
\sumd U 0 H_{\et}^{1}(x,\Lam) &\gets& \sumd U 1 H_{\et}^{2}(x,\Lam(1)) &\gets&
\sumd U 2 H_{\et}^{3}(x,\Lam(2)) &\gets& \cdots \\
\end{matrix}
\end{equation*}
\bigskip

Recall
$$H_a(\EUnz {\bullet} {-1} 1)=\EUnz {\bullet} {-1} 2=\KHetnz a U.$$
We are now interested in
$$
H_a(\EUnz {\bullet} {1} 1)=\EUnz {\bullet} {1} 2 \qaq
H_a(\EUnz {\bullet} {0} 1)=\EUnz {\bullet} {0} 2.
$$
Under the assumption of the Bloch-Kato conjecture,
$\EUnz {\bullet} {1} 1$ and $\EUnz {\bullet} {0} 1$ are identified with
the following complexes:
\begin{multline*}
C^1_{\bullet}(U,\nz)\;:\;
0 \gets \sumd U 1 \CH^1(x,2;\nz)  \gets \sumd U 2 \CH^2(x,3;\nz)\gets
\sumd U 3 \CH^3(x,4;\nz) \gets  \cdots ,\\
\end{multline*}
\begin{multline*}
C^0_{\bullet}(U,\nz)\;:\;
\sumd U 0 \nz \gets \sumd U 1 \CH^1(x,1,\nz)  \\
\gets \sumd U 2 \CH^2(x,2,\nz) \gets \sumd U 3 \CH^3(x,3,\nz) \gets \cdots ,
\end{multline*}
where the terms $\sumd U a$ are in degree $a$ and
$\CH^a(x,b;\nz)$ is Bloch's higher Chow group with finite coefficient.
More precisely, we have the following (see Theorem \ref{BL} in \S6):

\begin{lem}\label{BLEC}
There are natural map of complexes
$$
C^i_{\bullet}(U,\nz) \to \EUnz {\bullet} {i} 1\qfor i=0,1.
$$
The maps are isomorphism for the terms in degrees $\leq t$ if
$\BKlz t$ holds for all primes $l|n$.
\end{lem}

We note also that $C^0_{\bullet}(U,\nz)$ is isomorphic to the following
complex due to Nesterenko-Suslin \cite{NS} and Totaro \cite{To}

\begin{multline*}
\sumd U 0 \nz \gets \sumd U 1 K^M_1(\k(x))/n
\gets \sumd U 2 K^M_2(\k(x))/n \gets \sumd U 3 K^M_3(\k(x))/n \gets \cdots ,
\end{multline*}

Now the following result is an immediate consequence of \ref{mainthmIfinite}.

\begin{cor}\label{cormainthmIfinite}
Let $X$ be projective smooth of dimension $d$ over a finite field and
let $Y\subset X$ be a simple normal crossing divisor on $X$ such that
one of its irreducible components is an ample divisor. Put $U=X-Y$.
Let $n>1$ be an integer. Let $d=\dim(U)$.
\medbreak
\begin{itemize}
\item[(1)]
$H_0(C^0_{\bullet}(U,\nz))=CH^d(U)/n =0$ for $d\geq 2$.
\medbreak
\item[(2)]
$H_1(C^0_{\bullet}(U,\nz))=CH^d(U,1;\nz)=0$ for $d\geq 3$,
assuming $\BKlz 3$ for all primes $l|n$.
\medbreak
\item[(3)]
$H_2(C^0_{\bullet}(U,\nz))=0$ for $d\geq 4$,
assuming $\BKlz 4$ for all primes $l|n$.
\medbreak
\item[(4)]
$H_3(C^0_{\bullet}(U,\nz))\simeq H_1(C^1_{\bullet}(U,\nz))$ for $d\geq 5$,
assuming $\BKlz 5$ for all primes $l|n$ and
either of {$\bf(RS)_{d}$} or {$\bf(RES)_{3}$}.
\medbreak
\item[(5)]
$H_4(C^0_{\bullet}(U,\nz))\simeq H_2(C^1_{\bullet}(U,\nz))$ for $d\geq 6$,
assuming $\BKlz 6$ for all primes $l|n$ and
either of {$\bf(RS)_{d}$} or {$\bf(RES)_{4}$}.
\end{itemize}
\end{cor}
\medbreak


\bigskip

\section{\'Etale cycle map for motivic cohomology over finite fields}
\bigskip

In this section we give an application of the results in the previous section
to \'etale cycle map for motivic cohomology over finite fields.
First we recall briefly some fundamental facts on motivic cohomology.

Fix a base field $F$. Let $X$ be a quasi-projective scheme over $F$.
For an integer $i\geq 0$, let
$$ \Delta^q =\Spec(\bZ[t_0,\dots, t_q]/(\sum_{\nu=0}^q t_\nu \; -1)$$
be the algebraic $q$-simplex. We have Bloch's cycle complex (\cite{B1})
$$
z_s(X,\bullet)\;:\; \cdots \to
z_s(X,2) @>{\partial}>> z_s(X,1) @>{\partial}>> z_s(X,0).
$$
Here $z_s(X,q)$ is the free abelian group on closed integral subschemes of
dimension $s+q$ on $\Delta^q_X:=X\times \Delta^q$ which intersect all faces
properly where a face of $\Delta^q_X$ is a subscheme defined by an equation
$t_{i_1}=\cdots t_{i_e}=0$ for some $0\leq i_1<\cdots< i_e\leq q$.
The boundary maps of $z_s(X,\bullet)$ are given by taking the alternating sum
of the pullbacks of a cycle to the faces.
The complex $z_s(X,\bullet)$ is contravariant for flat morphisms
(with appropriate shift of degree) and covariant for proper morphisms.
The higher Chow groups of $X$ (resp. with finite coefficient for an integer
$n\geq 1$) are defined by
$$
\CH_s(X,q) = H_q(z_s(X,\bullet))\quad
(\text{resp. } \CH_s(X,q;\nz) = H_q(z_s(X,\bullet)\otimes^\bL \nz))).
$$
We have an exact sequence
\begin{equation}\label{mcfces}
0\to \CH_s(X,q)/n \to \CH_s(X,q;\nz) \to \CH_s(X,q-1)[n]\to 0.
\end{equation}
Assume now that $X$ is equi-dimensional and write
$$
\CH^r(X,q) = H_q(z^r(X,\bullet)),\quad
z^r(X,\bullet)=z_{\dim(X)-r}(X,\bullet).
$$
Assuming further that $X$ is smooth over $F$, the motivic cohomology of $X$
is defined as:
$$
  H^s_M(X,\bZ(r)) = \CH^r(X,2r-s)= H_{2r-s}(z^r(X,\bullet)).
$$
The finite-coefficient versions are also defined similarly. Note
$$ H^s_M(X,\bZ(r))= \CH^r(X,2r-s) =0 \qfor s>2r.$$
\medbreak

It is known (\cite{Ge2}, Lem.3.1) that the presheaves
$$ z^r(-,\bullet)\;:\; U \to z^r(U,\bullet) $$
are sheaves for the \'etale topology on $X$.
We define the complex $\bZ(r)_X$ of sheaves on the site $X_{Zar}$
as the cohomological complex with $z^r(-,2r-i)$ placed in degree $i$.
It is shown in \cite{B1} and \cite{Ge2}, Thm.3.2 that
$H^s_M(X,\bZ(r))$ agrees with $H^s_{Zar}(X,\bZ(r)_X)$,
the hypercohomology group of $\bZ(r)_X$.
We now recall the following result on the Beilinson-Lichtenbaum conjecture
due to Suslin-Voevodsky \cite{SV} and Geisser-Levine
\cite{GL2}, Thm.1.5 and \cite{GL1}, thm.8.5.
For an integer $n>0$, let $\nz(r)$ be the object of $D^b(X_{\et})$
defined in \eqref{Tatetwist}.

\begin{thm}\label{BL}
Let $X$ be a smooth scheme over $F$.
Let $\epsilon: X_{\et} \to X_{Zar}$ be the continuous map of sites.
\begin{itemize}
\item[(1)]
There is an \'etale cycle map
$$
cl^r_{\et}\;:\; \epsilon^* \bZ(r)_X\otimes^\bL \nz \to \nz(r),
$$
which is an isomorphism in $D^-(X_{\et})$, the derived category of
bounded-above complexes of \'etale sheaves on $X$.
\item[(2)]
The map $cl^r_{\et}$ induces a map
$$
\phi^r_X \;:\; \bZ(r)_X\otimes^\bL \nz \to \tau_{\leq r} R\epsilon_* \nz(r),
$$
which is an isomorphism if $\BKlz r$ holds for all primes $l|n$.
In particular it induces
$$
H^s_M(X,\nz(r))=\CH^r(X,2r-s;\nz) \simeq H^s_{\et}(X,\nz(r))
\qfor \forall s\leq r.
$$
\end{itemize}
\end{thm}
\medbreak

For $X$ smooth over $F$, we get by \ref{BL} the canonical map
from motivic cohomology to \'etale cohomology:
\begin{equation*}
\phi^{r,s}_X \;:\; H^s_M(X,\nz(r)) \to H^s_{\et}(X,\nz(r)).
\end{equation*}
We rewrite $\phi_X^{r,2r-q}$ by using higher Chow group as:
\begin{equation}\label{ccm}
\rho^{r,q}_X \;:\; \CH^r(X,q;\nz) \to H^{2r-q}_{\et}(X,\nz(r)).
\end{equation}

\begin{lem}\label{cyclekh}
Let $F$ be a finite field and $X$ be smooth of pure dimension $d$ over $F$.
Let $q\geq 0$ be an integer and assume $\BKlz {q+1}$ for all primes $l|n$.
If $r>d$, $\rho^{r,t}_X$ is an isomorphism for $\forall t\leq q$.
For $r=d$ there is a long exact sequence
\begin{multline*}
\KHetnz {q+2} X \to \CH^d(X,q;\nz) @>{\rho_X^{d,2d-q}}>>
H^{2d-q}_{\et}(X,\nz(d)) \\
\to \KHetnz {q+1} X \to \CH^d(X,q-1;\nz) @>{\rho_X^{d,2d-q+1}}>>
H^{2d-q+1}_{\et}(X,\nz(d)) \to \cdots
\end{multline*}
\end{lem}
\begin{proof}
Write $c=r-d$.
By the localization theorem for higher Chow groups (\cite{B2} and \cite{L}),
we have the niveau spectral sequence
$$
{^{CH}E}^1_{a,b}= \underset{x\in X_a}{\bigoplus} \CH^{a+c}(x,a+b;\nz)
\Rightarrow \CH^{r}(X,a+b;\nz).
$$
By the purity for \'etale cohomology, we have the niveau spectral sequence
$$
{^{\et}E}^1_{a,b}= \underset{x\in X_a}{\bigoplus} H^{a-b+2c}_{\et}(x,\nz(a+c))
\Rightarrow H^{2r-a-b}_{\et}(X,\nz(r)).
$$
The cyle map $\rho^{r,a+b}_X$ preserves the induced filtrations
and induces maps on $E^\infty_{a,b}$ compatible with
the cycle maps for $x\in X_a$:
$$
\rho_x^{a+c,a+b}\;:\; \CH^{a+c}(x,a+b;\nz) \to H^{a-b+2c}_{\et}(x,\nz(a+c)).
$$
By \ref{BL} $\BKlz {q+1}$ for all primes $l|n$ imply that
$\rho_x^{a+c,a+b}$ is an isomorphism if $b\geq c$ and $a+b \leq q+1$.
We note that ${^{CH}E}^1_{a,b}=0$ for $b<c$ and
$^{\et}E^1_{a,b}=0$ for $b < 2c-1$ since for $x\in X_a$,
$cd(\kappa(x))=a+1$ and $\DWtlog u x =0$ for $u>a$.
In case $c\geq 1$ it implies that $\rho^{r,a+b}_X$ induces
${^{CH}E}^\infty_{a,b} \simeq {^{\et}E}^\infty_{a,b}$ for $a+b\leq q$.
In case $c=0$ it implies that we have an exact sequence:
\begin{multline*}
{^{\et}E}^2_{q+2,-1} \to \CH^d(X,q;\nz) @>{\rho_X^{d,2d-q}}>>
H^{2d-q}_{\et}(X,\nz(d)) \\
\to {^{\et}E}^2_{q+1,-1} \to \CH^d(X,q-1;\nz) @>{\rho_X^{d,2d-q+1}}>>
H^{2d-q+1}_{\et}(X,\nz(d)) \to \cdots
\end{multline*}
This completes the proof of the lemma since $E^2_{a,-1}=\KHetnz {a} X$
by definition.
\end{proof}
\medbreak

Note that Theorem \ref{cor1intro} follows immediately from
Theorem \ref{mainthmKC1} and Lemma \ref{cyclekh}.

\begin{thm}\label{finiteness}
Let $F$ be a finite field of characteristic $p$. Let $X$ be a quasi-projective
equidimensional scheme of pure dimension $d$ over $F$.
\begin{itemize}
\item[(1)]
Assume $r>d$ and $\BKlz {q+1}$ for all primes $l|n$.
Then $\CH^r(X,t;\nz)$ is finite for $\forall t\leq q$.
\item[(2)]
Assume {$\bf(RS)_{d}$} and $\BKlz {q+2}$ for all primes $l|n$.
Then $\CH^d(X,t;\nz)$ is finite for $\forall t\leq q$.
\end{itemize}
\end{thm}
\begin{proof}
In fact we show the finiteness of $\CH_s(X,q;\nz)$ for $s<0$ in (1) and that
for $s=0$ in (2) without assuming that $X$ is equi-dimensional.
In case $X$ is smooth over $F$, \ref{finiteness} follows from
\ref{cyclekh} and \ref{mainthmKC2} in view of finiteness of
\'etale cohomology $H^s_{\et}(X,\nz(r))$ (For the prime-to-p part
it follows from SGA4$\frac{1}{2}$, Th. finitude. For the $p$-part we need
assume $r\geq d$ and it follows from a result of Moser \cite{Mo}).
We now proceed by the induction on $\dim(X)$.
Assume \ref{finiteness} is proved in dimension$<d$.
Take a closed subscheme $Z\subset X$ such that $U:=X-Z$ is smooth over $F$
and dense in $X$.
We have the localization exact sequence (\cite{B2} and \cite{L})
$$
\CH_s(Z,t;\nz) \to \CH_s(X,t;\nz) \to \CH_s(U,t;\nz).
$$
This completes the proof by induction.
$\square$
\end{proof}

\bigskip

\small{
\noindent
Uwe Jannsen\\
NWF I - Mathematik\\
Universit\"at Regensburg\\
Universit\"atsstr. 31\\
93053 Regensburg\\
Germany\\
email: uwe.jannsen@mathematik.uni-regensburg.de

\medskip

\noindent
Shuji Saito\\
Graduate School of Mathematics\\
University of Tokyo\\
Komaba, Meguro-ku\\
Tokyo 153-8914\\
Japan\\
email: sshuji@ms.u-tokyo.ac.jp
}


\end{document}